\documentclass[11pt]{article}
	
	\usepackage{caption}
    \usepackage{url}
    \usepackage{verbatim}
	\usepackage{amsmath}
    \usepackage{color}
	\usepackage{amsthm}
	\usepackage{amsfonts}
	\usepackage{graphicx}
	\usepackage{nicefrac}

	\def\captionfont{\setb@se{11pt}\protect\footnotesize}
    \def\captionfont{\protect\footnotesize}

    \newcommand{\iprd}[2]{\left( #1 , #2 \right)}

    \newcommand{\aiprd}[2]{a\left( #1 , #2 \right)}

    \newcommand{\Soh}{\mathring{S}_h}

	\def\norm#1#2{\left\| #1 \right\|_{#2}}

	\newcommand{\dtau}{\delta_\tau}
	\newcommand{\ddtau}{\delta_{\tau}^2}

	\newcommand{\phih}{\phi_h}
	
	\newcommand{\phitil}{\tilde{\phih}}
	\newcommand{\phich}{\check{\phih}}

	\newcommand{\muh}{\mu_h}

	\newcommand{\hf}{\frac{1}{2}}
	\newcommand{\eA}{\mathcal{E}_a}
	\newcommand{\eh}{\mathcal{E}_h}
	\newcommand{\e}{\mathcal{E}}
	\newcommand{\eAphi}{\mathcal{E}_a^{\phi}}
	\newcommand{\ehphi}{\mathcal{E}_h^{\phi}}
	\newcommand{\ephi}{\mathcal{E}^{\phi}}

	\newtheorem{thm}{Theorem}[section]

	\newtheorem{lem}[thm]{Lemma}
	\newtheorem{rem}[thm]{Remark}

\begin{document}
\title{Stability and Convergence of a Second Order Mixed Finite Element Method for the Cahn-Hilliard Equation}

	\author{
Amanda E. Diegel\thanks{Department of Mathematics, The University of Tennessee, Knoxville, TN 37996 (diegel@math.utk.edu)},
	\and
Cheng Wang\thanks{Department of Mathematics, The University of Massachusetts, North Dartmouth, MA 02747 (cwang1@umassd.edu)},
	\and
Steven M. Wise\thanks{Department of Mathematics, The University of Tennessee, Knoxville, TN 37996 (Corresponding author: swise@math.utk.edu)}}

	\maketitle
	
	\numberwithin{equation}{section}
	
	\begin{abstract}
In this paper we devise and analyze an unconditionally stable, second-order-in-time numerical scheme for the Cahn-Hilliard equation in two and three space dimensions.  We prove that our two-step scheme is unconditionally energy stable and unconditionally uniquely solvable. Furthermore, we show that the discrete phase variable is bounded in $L^\infty \left(0,T;L^\infty\right)$ and the discrete chemical potential is bounded in $L^\infty \left(0,T;L^2\right)$, for any time and space step sizes, in two and three dimensions, and for any finite final time $T$. We subsequently prove that these variables converge with optimal rates in the appropriate energy norms in both two and three dimensions. We include in this work a detailed analysis of the initialization of the two-step scheme.
	\end{abstract}
	

	\section{Introduction}

Let $\Omega\subset \mathbb{R}^d$, $d=2,3$, be an open polygonal or polyhedral domain. For all $\phi \in H^1(\Omega)$, consider the energy \cite{cahn58}
	\begin{align}
E(\phi) = \int_{\Omega} \left\{\frac{1}{4\varepsilon} \left(\phi^2 - 1\right)^2 + \frac{\varepsilon}{2} |\nabla \phi|^2 \right\} d\bf{x},
	\label{eq:Ginzburg-Landau-energy}
	\end{align}
where $\phi$ is the concentration field and $\varepsilon$ is a positive constant. The phase equilibria are represented by the values $\phi = \pm 1$. One version of the  celebrated Cahn-Hilliard equation is given by \cite{cahn61,cahn58}:
	\begin{subequations}
	\begin{align}
\partial_t \phi = \varepsilon \Delta \mu, & & &\text{in} \ \Omega_T ,
	\label{eq:CH-mixed-a-alt}
	\\
\mu = \varepsilon^{-1}\left(\phi^3-\phi\right)-\varepsilon \Delta \phi, & & &\text{in} \ \Omega_T,   
	\label{eq:CH-mixed-b-alt}
	\\
\partial_n \phi =\partial_n \mu = 0, & & &\text{on} \  \partial \Omega\times (0,T) ,
	\label{eq:CH-mixed-c-alt}
	\end{align}
	\end{subequations}
where $\mu := \delta_\phi E$ is the chemical potential. The boundary conditions represent local thermodynamic equilibrium ($\partial_n\phi =0$) and no-mass-flux ($\partial_n\mu =0$). Clearly $E(\phi) \ge 0$ for all $\phi\in H^1(\Omega)$.  Additionally, for all $\varepsilon >0$ and $\phi \in H^1(\Omega)$, there exist positive constants $K_1=K_1(\varepsilon)$ and $K_2=K_2(\varepsilon)$ such that
	\begin{equation} 
0 < K_1 \norm{\phi}{H^1}^2  \le E(\phi) + K_2.
	\end{equation}
	
A weak formulation of \eqref{eq:CH-mixed-a-alt} -- \eqref{eq:CH-mixed-c-alt} may be written as follows: find $(\phi,\mu)$ such that
	\[
\phi \in \, L^\infty\left(0,T;H^1(\Omega)\right)\cap  L^4\left(0,T;L^\infty(\Omega)\right),\ \partial_t \phi \in \, L^2\bigl(0,T; H^{-1}(\Omega)\bigr), \ \mu \in \, L^2\bigl(0,T;H^1(\Omega)\bigr),
	\]
and there hold for almost all $t\in (0,T)$ 
	\begin{subequations}
	\begin{align}
\langle \partial_t \phi ,\nu \rangle + \varepsilon \,\aiprd{\mu}{\nu} &= 0  &&\quad \forall \,\nu \in H^1(\Omega),
	\label{eq:weak-mch-a} 
	\\
\iprd{\mu}{\psi}-\varepsilon \,\aiprd{\phi}{\psi} - \varepsilon^{-1}\iprd{\phi^3-\phi}{\psi}  &= 0  &&\quad \forall \,\psi\in H^1(\Omega),
	\label{eq:weak-mch-b} 
	\end{align}
	\end{subequations}
where
	\begin{equation}
\aiprd{u}{v} := \iprd{\nabla u}{\nabla v}, 
	\end{equation}
with the ``compatible" initial data
	\begin{align}
\phi(0) = \phi_0 \in H^2_N(\Omega) &:= \left\{v\in H^2(\Omega) \,\middle| \,\partial_n v = 0 \,\mbox{on} \,\partial\Omega \right\}.
	\end{align}
Here we use the notations  $H^{-1}(\Omega) := \left(H^1(\Omega)\right)^*$ and $\langle  \, \cdot \, , \, \cdot \, \rangle$ as the duality paring between $H^{-1}$ and $H^1$.  Throughout the paper, we use the notation $\Phi(t) := \Phi(\, \cdot \, , t)\in X$, which views a spatiotemporal function as a map from the time interval $[0,T]$ into an appropriate Banach space, $X$.  The system \eqref{eq:weak-mch-a} -- \eqref{eq:weak-mch-b} is mass conservative:  for almost every $t\in[0,T]$, $\iprd{\phi(t)-\phi_0}{1} = 0$.  This observation rests on the fact that $\aiprd{\phi}{1} = 0$, for all $\phi\in L^2(\Omega)$.  Observe that the homogeneous Neumann boundary conditions associated with the phase variables $\phi$ and $\mu$ are natural in this mixed weak formulation of the problem.  

The existence of weak solutions is a straightforward exercise using the compactness/energy method, for example,~\cite{elliott86}. It is likewise straightforward to show that weak solutions of \eqref{eq:weak-mch-a} -- \eqref{eq:weak-mch-b} dissipate the energy \eqref{eq:Ginzburg-Landau-energy}. In other words, \eqref{eq:CH-mixed-a-alt} -- \eqref{eq:CH-mixed-c-alt} is a mass-conservative gradient flow with respect to the energy \eqref{eq:Ginzburg-Landau-energy}. Precisely, for any $t\in[0,T]$, we have the energy law
	\begin{equation}
E(\phi(t)) +\int_0^t \varepsilon\norm{\nabla\mu(s)}{L^2}^2 ds = E(\phi_0) .
	\label{eq:pde-energy-law}
	\end{equation}
	
The Cahn-Hilliard equation is one of the most important models in mathematical physics. On its own, the equation is a model for spinodal decomposition \cite{cahn61}. However, the Cahn-Hilliard equation is more often paired with equations that describe important physical behavior of a given physical system, typically through nonlinear coupling terms. Prominent examples include the Cahn-Hilliard-Navier-Stokes equation, describing two-phase flow~\cite{feng06, grun13, grun14, kay07, liu03, shen10b}, the Cahn-Hilliard-Hele-Shaw equation \cite{lee02a, lee02b, wise10} which describes spinodal decomposition of a binary fluid in a Hele-Shaw cell, and the Cahn-Larch\'{e} equation~\cite{fratzl99, garcke05, larche82, wise05} describing solid-state, binary phase transformations involving coherent, linear-elastic misfit.

The Cahn-Hilliard equation is a challenging fourth-order, nonlinear parabolic-type partial differential equation.  Naive explicit methods suffer from severe time-step restrictions for stability.  On the other hand, fully implicit numerical methods must contend with a potentially large nonlinear system of algebraic equations.  There remains a great need for sophisticated stable and efficient numerical schemes for the Cahn-Hilliard equation.  Indeed, extensive research has been conducted in this area, in particular for first-order-accurate-in-time schemes, see \cite{aristotelous13, chen98, elliott92, elliott93, elliott96, feng06, feng04, furihata01, guan13, he07, kay06, kay07, kim03, wise10} and the references therein.  Less commonly investigated are second-order-accurate-in-time numerical schemes.  In general, the analysis of second-order schemes for nonlinear equations can be significantly more difficult than that for first-order methods.  Nevertheless, such work has been reported in the following articles~\cite{aristotelous14, chen98, du91, elliott89a, furihata01, shen12, shen10a, wu14}. We mention, in particular, the secant-type algorithms described in~\cite{du91, furihata01}. With the notation $\Psi(\phi) := \frac{1}{4} \left(\phi^2 - 1\right)^2$, the secant  scheme of~\cite{du91} for the Cahn-Hilliard equation may be formulated as 
	\begin{align}
\phi^{n+1} - \phi^n = s \varepsilon \Delta \mu^{n+\frac12}, \quad \mu^{n+\frac12} := \varepsilon^{-1}\frac{\Psi(\phi^{n+1}) - \Psi(\phi^n)}{\phi^{n+1} - \phi^n} - \frac{\varepsilon}{2} \left(\Delta \phi^{n+1} - \Delta \phi^n\right).
	\end{align}
This scheme is energy stable. However, it may not be unconditionally uniquely solvable with respect to the time step size $s$. (See \cite{du91, elliott89a, furihata01} for details.) Lack of unconditional solvability may be problematic as coarsening studies using the Cahn Hilliard equation may involve very large time scales, requiring potentially very large time steps for efficiency.

Chen and Shen introduce a semi-implicit Fourier-spectral method in \cite{chen98} which has a couple of advantages over explicit Euler finite difference methods. In their scheme, the high-order semi-implicit treatment in time enables the use of larger time steps while maintaining higher accuracy. However, even though the time step size may be taken to be larger, the scheme's stability is still not completely independent on the time step size. Furthermore, although they test their scheme through numerical simulations, no formal stability or convergence analyses are presented in the paper. 

Wu, Zwieten, and Van Der Zee \cite{wu14} introduce a semi-discrete second-order convex-splitting scheme for Cahn-Hilliard-type equations with applications to diffuse-interface tumor-growth models. They are able to show unconditional energy stability relative to the energy norms, mass conservation, and a second order local truncation error for the phase field parameter. However, they do not prove second order accuracy relative to the energy norm for the phase field parameter. 

In contrast to the papers referenced above, we propose a new second-order-accurate-in-time, fully discrete, mixed finite element scheme for the Cahn-Hilliard problem \eqref{eq:CH-mixed-a-alt} -- \eqref{eq:CH-mixed-c-alt}, which is closely related to the finite difference scheme proposed in \cite{guo14}:
	\begin{subequations}	
	\begin{align}
\phih^{n+1} - \phih^n &= s \, \varepsilon \, \Delta_h \muh^{n+\hf},
	\label{eq:scheme-a-pre}
	\\
\muh^{n+\hf} &:= \frac{1}{4 \varepsilon} \left(\phih^{n+1} + \phih^n\right) \left(\left(\phih^{n+1}\right)^2 + \left(\phih^n\right)^2\right)  - \frac{1}{\varepsilon} \left(\frac32 \phih^{n} - \frac12 \phih^{n-1} \right)
	\nonumber
	\\
&- \varepsilon \Delta_h \left(\frac34 \phih^{n+1} + \frac14 \phih^{n-1}\right),
	\label{eq:scheme-b-pre}
	\end{align}
	\end{subequations}
where $\Delta_h$ above is a finite difference stencil approximating the Laplacian, and $\phi_h$ and $\mu_h$ are grid variables. The formulation of the scheme \eqref{eq:scheme-a-pre} -- \eqref{eq:scheme-b-pre} uses a convex splitting of the energy~\cite{eyre98, elliott89b, feng12, wise09a}. Observe  that the energy \eqref{eq:Ginzburg-Landau-energy} may be represented as the difference between two purely convex energies: 
	\begin{align}
E(\phi) = E_c(\phi) - E_e(\phi) = \frac{1}{4\varepsilon}\norm{\phi}{L^4}^4 +\frac{\varepsilon}{2} \norm{\nabla \phi}{L^2}^2 + \frac{|\Omega|}{4\varepsilon} - \frac{1}{2\varepsilon}\norm{\phi}{L^2}^2 .
	\label{eq:continuous-energy}
	\end{align}
The idea is then to treat the variation of $E_c$ implicitly and that of $E_e$, explicitly. The advantages of the scheme \eqref{eq:scheme-a-pre} -- \eqref{eq:scheme-b-pre} are three fold. The scheme is unconditionally energy stable, unconditionally uniquely solvable, and converges optimally in the energy norm. In our finite element version of the scheme, the stability and solvability statements we prove are \emph{completely unconditional with respect to the time and space step sizes}. In fact, \emph{all} of our \emph{a priori} stability estimates hold completely independently of the time and space step sizes.  We use a bootstrapping technique to leverage the energy stabilities to achieve unconditional $L^{\infty}(0,T; L^{\infty}(\Omega))$ stability for the phase field variable $\phi_h$ and unconditional $L^{\infty}(0,T; L^2(\Omega))$ stability for the chemical potential $\mu_h$.  With these stabilities in hand, we are then able to prove optimal error estimates for $\phi_h$ and $\mu_h$ in the appropriate energy norms.

The remainder of the paper is organized as follows. In Section~\ref{sec:defn-and-properties}, we define our second-order mixed finite element version of the scheme and prove the unconditional solvability and stability. In Section~\ref{sec:error estimates}, we prove error estimates for the scheme under suitable regularity assumptions for the PDE solution. In Section~\ref{sec:numerical-experiments}, we present the results of numerical tests that confirm the rates of convergence predicted by the error estimates.

	\section{A Mixed Finite Element Convex Splitting Scheme}
	\label{sec:defn-and-properties}
	
	\subsection{Definition of the Scheme}
	\label{subsec-defn}

Let $M$ be a positive integer and $0=t_0 < t_1 < \cdots < t_M = T$ be a uniform partition of $[0,T]$, with $\tau = t_i-t_{i-1}$ and $i=1,\ldots ,M$.  Suppose ${\mathcal T}_h = \left\{ K \right\}$ is a conforming, shape-regular, quasi-uniform family of triangulations of $\Omega$.  For $q\in\mathbb{Z}^+$, define $S_h := \left\{v\in C^0(\Omega) \, \middle| \,v|_K \in {\mathcal P}_q(K), \,\forall \,\,  K\in \mathcal{T}_h \right\}\subset H^1(\Omega)$. Define $\Soh := S_h\cap L_0^2(\Omega)$, with $L_0^2(\Omega)$ denoting those functions in $L^2(\Omega)$ with zero mean. Our mixed second-order splitting scheme is defined as follows:  for any $1\le m\le M-1$, given  $\phih^{m}, \phih^{m-1} \in S_h$, find $\phih^{m+1},\muh^{m+\frac12} \in S_h$ such that  
	\begin{subequations}
	\begin{align}
\iprd{\dtau \phih^{m+\frac12}}{\nu} + \varepsilon \,\aiprd{\muh^{m+\frac12}}{\nu} &= \, 0  & \forall \, \nu \in S_h ,
	\label{eq:scheme-a}
	\\
\varepsilon^{-1} \,\iprd{\chi\left(\phih^{m+1},\phih^m\right)}{\psi} - \varepsilon^{-1} \iprd{\phitil^{m+\frac12}}{\psi}  & &
	\nonumber
	\\
+ \varepsilon \,\aiprd{\phich^{m+\frac12}}{\psi}- \iprd{\muh^{m+\frac12}}{\psi} &= \, 0 & \forall \, \psi\in S_h,
	\label{eq:scheme-b}
	\end{align}
	\end{subequations}
where
	\begin{align}
\dtau \phih^{m+\frac12} &:= \frac{\phih^{m+1} - \phih^{m}}{\tau}, \quad \phih^{m+\frac12} := \frac12 \phih^{m+1} + \frac12 \phih^m, \quad \phitil^{m+\frac12} := \frac32 \phih^m - \frac12 \phih^{m-1},
	\\
\phich^{m+\frac12} &:= \frac34 \phih^{m+1} + \frac14 \phih^{m-1}, \quad  \chi\left(\phih^{m+1},\phih^m\right) := \frac{1}{2}\left(\left(\phih^{m+1}\right)^2 + \left(\phih^{m}\right)^2\right)\phi^{m+\frac12}.
	\end{align}
Since this is a multi-step scheme, it requires a separate initialization process. For the first step, the scheme is as follows: given $\phih^0 \in S_h$, find $\phih^1, \muh^{\frac12} \in S_h$ such that
	\begin{subequations}
	\begin{align}
\iprd{\dtau \phih^{\frac12}}{\nu} + \varepsilon \,\aiprd{\muh^{\frac12}}{\nu} &= \, 0  & \forall \, \nu \in S_h ,
	\label{eq:scheme-a-initial}
	\\
\varepsilon^{-1} \,\iprd{\chi\left(\phih^{1},\phih^0\right)}{\psi} - \varepsilon^{-1} \iprd{\phih^0}{\psi} + \frac{\tau}{2} \, \aiprd{\muh^0}{\psi}  & &
	\nonumber
	\\
+ \varepsilon \,\aiprd{\phih^{\frac12}}{\psi} - \iprd{\muh^{\frac12}}{\psi} &= \, 0  & \forall \, \psi\in S_h,
	\label{eq:scheme-b-initial}
	\end{align}
	\end{subequations}
where $\phih^0 := R_h \phi_0$, and the operator $R_h: H^1(\Omega) \to S_h$ is a standard Ritz projection:
	\begin{equation}
\aiprd{R_h\phi - \phi}{\xi} = 0 \quad \forall \, \xi\in S_h, \quad \iprd{R_h \phi-\phi}{1}=0.
	\label{eq:Ritz-projection}
	\end{equation}
Note that the scheme requires initial data for the chemical potential, $\muh^0\in S_h$, which is defined as $\muh^0 := R_h \mu_0$, where
	\begin{equation}
\mu_0 := \varepsilon^{-1}\left(\phi_0^3 - \phi_0\right) - \varepsilon\Delta\phi_0.
	\end{equation}

	\begin{thm}
The scheme \eqref{eq:scheme-a} -- \eqref{eq:scheme-b} coupled with the initial scheme \eqref{eq:scheme-b-initial} -- \eqref{eq:scheme-b-initial} is uniquely solvable for any mesh parameters $h$ and $\tau$ and for any model parameters.
	\end{thm}
	
	\begin{proof}
The proof is based on convexity arguments and follows in a similar manner as that of Theorem 5 from reference \cite{hu09}. We omit the details for brevity.
	\end{proof}

	\begin{rem}
Note that it is not necessary for solvability and some basic energy stabilities that the $\mu$--space and the $\phi$--space be equal.  However, the proofs of the higher-order stability estimates, in particular the proof in Lemma~\ref{lem-a-priori-stability-dtau-mu}, do require the equivalence of these spaces.
	\end{rem}
	
	\begin{rem}
	\label{rem:initial-projection}
The elliptic projections are used in the initialization for simplicity in the forthcoming error analysis. However, other (simpler) projections may be used in the initialization step, as long as they have good approximation properties.
	\end{rem}
	
\subsection{Unconditional Energy Stability}
	\label{subsec-energy-stability}

We now show that the solutions to our scheme enjoy stability properties that are similar to those of the PDE solutions, and moreover, these properties hold regardless of the sizes of $h$ and $\tau$.  The first property, the unconditional energy stability, is a direct result of the convex decomposition. We begin the discussion with the definition of the discrete Laplacian, $\Delta_h: S_h \to \Soh$, as follows:  for any $v_h\in S_h$, $\Delta_h v_h\in \Soh$ denotes the unique solution to the problem
	\begin{equation}
\iprd{\Delta_h v_h}{\xi} = -\aiprd{v_h}{\xi}  \quad\forall \, \,\xi\in S_h.
	\label{eq:discrete-Laplacian}
	\end{equation}
In particular, setting $\xi = \Delta_h v_h$ in \eqref{eq:discrete-Laplacian}, we obtain  
	\begin{displaymath}
\norm{\Delta_h v_h}{L^2}^2 = -\aiprd{v_h}{ \Delta_hv_h} .
	\end{displaymath}

\begin{lem}
	\label{lem-energy-law-initial}
Let $(\phih^{1}, \muh^{\frac12}) \in S_h\times S_h$ be the unique solution of the initialization scheme \eqref{eq:scheme-a-initial} -- \eqref{eq:scheme-b-initial}.  Then the following first-step energy stability holds for any $h, \, \tau >0$:
	\begin{align}
E\left(\phih^{1}\right) + \tau \varepsilon \norm{\nabla\muh^{\frac12}}{L^2}^2 + \frac{1}{4\varepsilon} \norm{\phih^1 - \phih^0}{L^2}^2 \le E\left(\phih^0\right) + \frac{ \varepsilon \tau^2}{4} \norm{\Delta_h \muh^0}{L^2}^2,
	\label{eq:ConvSplitEnLaw-initial}
	\end{align}
where $E(\phi)$ is defined in \eqref{eq:continuous-energy}.
	\end{lem}

	\begin{proof}
Setting $\nu= \tau \muh^{\frac12}$ in \eqref{eq:scheme-a-initial} and $\psi = \tau \dtau \phih^{\frac12} = \phih^1 - \phih^0$ in \eqref{eq:scheme-b-initial} yields the following:
	\begin{align}
\tau \iprd{\dtau \phih^{\frac12}}{\muh^{\frac12}} + \tau \varepsilon \norm{\nabla\muh^{\frac12}}{L^2}^2 &=   0,
	\label{eq:tested-energy-1-initial}
	\\
	\nonumber
\varepsilon^{-1} \,\iprd{\chi\left(\phih^{1},\phih^0\right)}{\phih^1 - \phih^0} - \varepsilon^{-1} \iprd{\phih^{0}}{\phih^1 - \phih^0} + \varepsilon \,\aiprd{\phih^{\frac12}}{\phih^1 - \phih^0} \,&
	\\
+ \, \frac{\tau}{2} \, \aiprd{\muh^0}{\phih^1 - \phih^0} - \tau \iprd{\muh^{\frac12}}{\dtau \phih^{\frac12}} &= 0.
	\label{eq:tested-energy-2-initial}
	\end{align}
Adding Eqs.~\eqref{eq:tested-energy-1-initial} and \eqref{eq:tested-energy-2-initial}, using Young's inequality, and the following identities 
	\begin{align}
 \iprd{\chi\left(\phih^{1},\phih^0\right)}{\phih^1 - \phih^0} =& \, \frac14 \left(\norm{\phih^1}{L^4}^4 - \norm{\phih^0}{L^4}^4 \right),
	\label{eq:identity-nonlinear-initial}
	\\
\iprd{\phih^0}{\phih^1 - \phih^0} =& \, \frac12 \left(\norm{\phih^1}{L^2}^2 - \norm{\phih^0}{L^2}^2 - \norm{\phih^1 - \phih^0}{L^2}^2 \right),
	\label{eq:identity-linear-initial}
	\end{align}
the result is obtained.
	\end{proof}

We now define a modified energy
	\begin{equation}
F(\phi, \psi) := E(\phi) + \frac{1}{4\varepsilon} \norm{\phi - \psi}{L^2}^2 + \frac{\varepsilon}{8} \norm{\nabla \phi - \nabla \psi}{L^2}^2,
	\end{equation}
where $E(\phi)$ is defined as above.
	
	\begin{lem}
	\label{lem-energy-law}
Let $(\phih^{m+1}, \muh^{m+\frac12}) \in S_h\times S_h$ be the unique solution of  \eqref{eq:scheme-a} -- \eqref{eq:scheme-b}, and $(\phih^{1}, \muh^{\frac12}) \in S_h\times S_h$, the unique solution of \eqref{eq:scheme-a-initial} -- \eqref{eq:scheme-b-initial}.  Then the following energy law holds for any $h,\,  \tau >0$:
	\begin{align}
F\left(\phih^{\ell+1}, \phih^{\ell}\right) +\tau \varepsilon \sum_{m=1}^\ell \norm{\nabla\muh^{m+\frac12}}{L^2}^2 &+  \sum_{m=1}^\ell \Bigg[ \frac{1}{4\varepsilon} \norm{\phih^{m+1} - 2 \phih^m + \phih^{m-1}}{L^2}^2  
	\nonumber
	\\
&+ \frac{\varepsilon}{8} \norm{\nabla \phih^{m+1} - 2 \nabla \phih^m + \nabla \phih^{m-1}}{L^2}^2 \Biggr] = F\left(\phih^1, \phih^0\right),
	\label{eq:ConvSplitEnLaw}
	\end{align}
for all $1 \leq \ell \leq M-1$.
	\end{lem}

	\begin{proof}
Setting $\nu= \muh^{m+\frac12}$ in \eqref{eq:scheme-a} and $\psi = \dtau \phih^{m+\frac12}$ in \eqref{eq:scheme-b} gives
	\begin{align}
\iprd{\dtau \phih^{m+\frac12}}{\muh^{m+\frac12}} + \varepsilon \norm{\nabla\muh^{m+\frac12}}{L^2}^2 &=   0,
	\label{eq:tested-energy-1}
	\\
	\nonumber
 \varepsilon^{-1} \,\iprd{\chi\left(\phih^{m+1},\phih^m\right)}{\dtau \phih^{m+\frac12}} -  \varepsilon^{-1} \,\iprd{\phitil^{m+\frac12}}{\dtau \phih^{m+\frac12}} \quad &
	\\
+ \, \varepsilon \,\aiprd{\phich^{m+\frac12}}{\dtau \phih^{m+\frac12}} - \iprd{\muh^{m+\frac12}}{\dtau \phih^{m+\frac12}} &= 0.
	\label{eq:tested-energy-2}
	\end{align}
Combining \eqref{eq:tested-energy-1} -- \eqref{eq:tested-energy-2}, using the identities 
	\begin{align}
 \iprd{\chi\left(\phih^{m+1},\phih^m\right)}{\dtau \phih^{m+\frac12}} - \iprd{\phitil^{m+\frac12}}{\dtau \phih^{m+\frac12}}  = &\, \frac1{4\tau} \left( \norm{\left(\phih^{m+1}\right)^2-1}{L^2}^2 - \norm{\left(\phih^{m}\right)^2-1}{L^2}^2\right)
	\nonumber
	\\
&  + \frac1{4\tau} \left( \norm{\phih^{m+1} - \phih^m}{L^2}^2 - \norm{\phih^{m} - \phih^{m-1}}{L^2}^2\right)
	\nonumber
	\\
& + \frac1{4\tau} \norm{\phih^{m+1} - 2 \phih^m + \phih^{m-1}}{L^2}^2
	\label{eq:identity-nonlinear}
	\end{align}
and
	\begin{align}
\aiprd{\phich^{m+\frac12}}{\dtau\phih^{m+\frac12}}  = & \, \frac{1}{2 \tau} \left( \norm{\nabla \phih^{m+1}}{L^2}^2 - \norm{\nabla \phih^{m}}{L^2}^2 \right) 
	\nonumber
	\\
& + \frac{1}{8 \tau} \left(\norm{\nabla \phih^{m+1} - \nabla \phih^m}{L^2}^2 - \norm{\nabla \phih^{m} - \nabla \phih^{m-1}}{L^2}^2 \right)
	\nonumber
	\\
& + \frac{1}{8 \tau} \norm{\nabla \phih^{m+1} - 2 \nabla \phih^m + \nabla \phih^{m-1}}{L^2}^2,
	\label{eq:identity-check}
	\end{align}
and applying the operator $\tau\sum_{m=1}^\ell$ to the combined equation result in \eqref{eq:ConvSplitEnLaw}.
	\end{proof} 

In the sequel, we will make the following stability assumptions for the initial data:
	\begin{equation}
E\left(\phih^0\right) + \tau^2 \norm{\Delta_h \muh^0}{L^2}^2 + \norm{\Delta_h \phih^0}{L^2}^2 \le C,
	\label{eq:initial-stability}
	\end{equation}
for some constant $C>0$ that is independent of $h$ and $\tau$. Here we assume that $\varepsilon>0$ is fixed.  In fact, from this point in the stability and error analyses, we will not track the dependence of the estimates on the interface parameter $\varepsilon$, though this may be of importance, especially if $\varepsilon$ tends to zero.

	\begin{lem}
	\label{lem-a-priori-stability-trivial}
Let $(\phih^{m+1}, \muh^{m+\frac12}) \in S_h\times S_h$ be the unique solution of  \eqref{eq:scheme-a} -- \eqref{eq:scheme-b}, and $(\phih^1, \muh^{\frac12}) \in S_h \times S_h$, the unique solution of \eqref{eq:scheme-a-initial} -- \eqref{eq:scheme-b-initial}.  Then the following estimates hold for any $h,\, \tau>0$:
	\begin{align}
\max_{0\leq m\leq M} \left[ \norm{\nabla\phih^m}{L^2}^2 + \norm{\left( \phih^m\right)^2-1}{L^2}^2 \right] &\leq C, 
	\label{eq:Linf-phi-discrete}
	\\
\max_{0\leq m\leq M}\left[\norm{\phih^m}{L^4}^4 +\norm{\phih^m}{L^2}^2  +\norm{\phih^m}{H^1}^2 \right] &\leq C,
	\label{eq:Linf-H1-phi-discrete}
	\\
\max_{1\leq m\leq M}\left[\norm{\phih^m-\phih^{m-1}}{L^2}^2 + \norm{\nabla \phih^m - \nabla \phih^{m-1}}{L^2}^2\right] &\leq C,
	\label{eq:Linf-phi-diff-discrete}
	\\
\tau \sum_{m=0}^{M-1} \norm{\nabla\muh^{m+\frac12}}{L^2}^2  &\leq C , 
	\label{eq:sum-mu-discrete}
	\\
\sum_{m=1}^{M-1} \left[ \norm{\phih^{m+1} - 2 \phih^m + \phih^{m-1}}{L^2}^2 + \norm{\nabla \phih^{m+1} - 2 \nabla \phih^m + \nabla \phih^{m-1}}{L^2}^2 \right] &\leq C , 
	\label{eq:sum-phi-discrete}
	\end{align}
for some constant $C>0$ that is independent of $h$, $\tau$, and $T$.
	\end{lem}
	
	\begin{proof}
Starting with the stability of the initial step, inequality \eqref{eq:ConvSplitEnLaw-initial}, and considering the stability of the initial data, inequality~\eqref{eq:initial-stability}, we immediately have
	\begin{equation}
\norm{\nabla\phih^1}{L^2}^2 + \norm{\left( \phih^1\right)^2-1}{L^2}^2 +\norm{\phih^1}{L^4}^4 +\norm{\phih^1}{L^2}^2  +\norm{\phih^1}{H^1}^2 + \tau\norm{\nabla\muh^{\frac12}}{L^2}^2 \leq C .
	\end{equation}
The triangle inequality immediately implies
	\begin{align*}
F\left(\phih^1, \phih^0\right) = E(\phih^1) + \frac{1}{4\varepsilon} \norm{\phih^1 - \phih^0}{L^2}^2 + \frac{\varepsilon}{8} \norm{\nabla \phih^1 - \nabla \phih^0}{L^2}^2 \le C.
	\end{align*}
This, together with \eqref{eq:ConvSplitEnLaw} and the fact that $F(\phih^{m+1},\phih^m) \ge E(\phih^{m+1})$, for all $0\le m \le M-1$, establishes all of the inequalities.
	\end{proof}
	
We are able to prove the next set of \emph{a priori} stability estimates without any restrictions on $h$ and $\tau$.  See \cite{diegel14} for a definition of discrete negative norm $\norm{ \, \cdot\, }{-1,h}$.
	
	\begin{lem}
	\label{lem-improved-a-priori-stabilities}
Let $(\phih^{m+1}, \muh^{m+\frac12}) \in S_h\times S_h$ be the unique solution of  \eqref{eq:scheme-a} -- \eqref{eq:scheme-b}, and $(\phih^1, \muh^{\frac12}) \in S_h \times S_h$, the unique solution of \eqref{eq:scheme-a-initial} -- \eqref{eq:scheme-b-initial}. Then the following estimates hold for any $h,\, \tau >0$:
	\begin{align}
\tau \sum_{m=0}^{M-1} \Bigg[ \norm{\dtau \phih^{m+\frac12}}{H^{-1}}^2 +\norm{\dtau \phih^{m+\frac12}}{-1,h}^2 \Bigg] &\le C,
	\label{eq:sum-3D-good-dtau}
	\\
\tau \sum_{m=0}^{M-1}  \norm{\muh^{m+\frac12}}{L^2}^2 &\le C (T+1),
	\label{eq:sum-3D-good-mu}
	\\
\tau \sum_{m=1}^{M-1} \Bigg[	\norm{ \Delta_h \phich^{m+\frac12}}{L^2}^2  + \norm{\phich^{m+\frac12}}{L^\infty}^{\frac{4(6-d)}{d}} \Bigg] &\le C (T+1),
	\label{eq:sum-3D-good-phicheck}
	\end{align}
for some constant $C>0$ that is independent of $h$, $\tau$, and $T$.
	\end{lem}

	\begin{proof}
Let $\mathcal{Q}_h: L^2(\Omega) \to S_h$ be the $L^2$ projection, \emph{i.e.}, $\iprd{\mathcal{Q}_h \nu-\nu}{\xi}=0$ for all $\xi\in S_h$. Suppose $\nu\in\mathring{H}^1(\Omega)$. Then, by \eqref{eq:scheme-a} and \eqref{eq:scheme-a-initial}, for all $0 < m < M-1$
	\begin{align}
\iprd{\dtau\phih^{m+\frac12}}{\nu}= \iprd{\dtau\phih^{m+\frac12}}{\mathcal{Q}_h\nu}  =-\varepsilon \iprd{\nabla \muh^{m+\frac12}}{\nabla \mathcal{Q}_h\nu} &\leq \varepsilon\norm{\nabla\muh^{m+\frac12}}{L^2} \norm{\nabla \mathcal{Q}_h\nu}{L^2} 
	\nonumber
	\\
&\leq C \varepsilon\norm{\nabla\muh^{m+\frac12}}{L^2} \norm{\nabla \nu}{L^2},
	\end{align}	
where we used the $H^1$ stability of the $L^2$ projection in the last step. Applying $\tau\sum_{m=0}^{M-1}$ and using \eqref{eq:sum-mu-discrete} we obtain (\ref{eq:sum-3D-good-dtau}.1) -- which, in our notation, is the bound on the first term of the left side of \eqref{eq:sum-3D-good-dtau}. The estimate (\ref{eq:sum-3D-good-dtau}.2) follows from the inequality $\norm{\nu}{-1,h} \le \norm{\nu}{H^{-1}}$, which holds for all $\nu\in L^2(\Omega)$.

To prove \eqref{eq:sum-3D-good-mu}, for $1 \le m \le M-1$ we set $\psi = \muh^{m+\frac12}$ in \eqref{eq:scheme-b} to obtain
	\begin{align*}
\norm{\muh^{m+\frac12}}{L^2}^2 =& \, \varepsilon^{-1} \iprd{\chi\left(\phih^{m+1},\phih^m\right)}{\muh^{m+\frac12}} -  \varepsilon^{-1} \iprd{\phitil^{m+\frac12}}{\muh^{m+\frac12}} + \varepsilon \, \aiprd{\phich^{m+\frac12}}{\muh^{m+\frac12}} 
	\\
\le& \, C \norm{\chi\left(\phih^{m+1},\phih^m\right)}{L^2}^2 + \frac14 \norm{\muh^{m+\frac12}}{L^2}^2 + C \norm{\phitil^{m+\frac12}}{L^2}^2 + \frac14 \norm{\muh^{m+\frac12}}{L^2}^2 
	\\
&+ C \norm{\nabla \phich^{m+\frac12}}{L^2}^2 + \frac12 \norm{\nabla \muh^{m+\frac12}}{L^2}^2.
	\end{align*}
And, similarly, setting $\psi = \muh^{\frac12}$ in \eqref{eq:scheme-b-initial}, we have
	\begin{align*}
\norm{\muh^{\frac12}}{L^2}^2 \le& \, C \norm{\chi\left(\phih^{1},\phih^0\right)}{L^2}^2 + \frac16 \norm{\muh^{\frac12}}{L^2}^2 + C \norm{\phih^0}{L^2}^2 + \frac16 \norm{\muh^{\frac12}}{L^2}^2 + C \norm{\nabla \phih^{\frac12}}{L^2}^2 
	\\
&+ \frac12 \norm{\nabla \muh^{\frac12}}{L^2}^2 + \frac16 \norm{\muh^{\frac12}}{L^2}^2 + C \tau^2 \norm{\Delta_h \muh^0}{L^2}^2.
	\end{align*}
Hence, using the triangle inequality, \eqref{eq:Linf-H1-phi-discrete}, and the initial stability \eqref{eq:initial-stability}, we have for all $0 \le m \le M-1$, 
	\begin{align*}
\frac12 \norm{\muh^{m+\frac12}}{L^2}^2 \le& \, C \norm{\chi\left(\phih^{m+1},\phih^m\right)}{L^2}^2 + \frac12 \norm{\nabla \muh^{m+\frac12}}{L^2}^2 + C.
	\end{align*}
Now, using Lemma \ref{lem-a-priori-stability-trivial}, we have the following bound for all $0 \le m \le M-1$
	\begin{align}
\norm{\chi\left(\phih^{m+1},\phih^m\right)}{L^2}^2 =& \,  \frac{1}{16}\norm{\left(\phih^{m+1}\right)^3 + \left(\phih^{m+1}\right)^2 \phih^m + \phih^{m+1} \left(\phih^m\right)^2 + \left(\phih^{m}\right)^3}{L^2}^2
	\nonumber
	\\
\le& \, C \norm{\left(\phih^{m+1}\right)^3}{L^2}^2 + C \norm{\left(\phih^{m+1}\right)^2 \phih^m}{L^2}^2 + C \norm{\phih^{m+1} \left(\phih^m\right)^2}{L^2}^2 + C \norm{\left(\phih^{m}\right)^3}{L^2}^2
	\nonumber
	\\
\le& \, C \norm{\phih^{m+1}}{L^6}^6 + C \norm{\phih^m}{L^6}^6 \le C \norm{\phih^{m+1}}{H^1}^6 + C \norm{\phih^{m}}{H^1}^6 \le  C,
	\label{eq:chi-L2-stability}
	\end{align}
where we used Young's inequality and the embedding $H^1(\Omega) \hookrightarrow L^6(\Omega)$, for $d = 2, 3$. Hence,
	\begin{equation}
\norm{\muh^{m+\frac12}}{L^2}^2 \le \norm{\nabla\muh^{m+\frac12}}{L^2}^2 + C.
	\end{equation}
Applying $\tau\sum_{m=0}^{M-1}$, estimate \eqref{eq:sum-3D-good-mu} now follows from (\ref{eq:sum-mu-discrete}).

Setting $\psi_h= \Delta_h \phich^{m+\frac12}$ in \eqref{eq:scheme-b} and using the  definition of the discrete Laplacian \eqref{eq:discrete-Laplacian}, it follows that for all $1 \le m \le M-1$
	\begin{align*}
\varepsilon \norm{\Delta_h \phich^{m+\frac12}}{L^2}^2  =& \, -  \varepsilon \,\aiprd{\phich^{m+\frac12}}{ \Delta_h \phich^{m+\frac12}}
	\\
=& \, -\iprd{\muh^{m+\frac12}}{\Delta_h \phich^{m+\frac12}} - \varepsilon^{-1} \, \iprd{\phitil^{m+\frac12}}{\Delta_h \phich^{m+\frac12}} 
	\\
&+ \varepsilon^{-1} \,\iprd{\chi\left(\phih^{m+1},\phih^m\right)}{\Delta_h \phich^{m+\frac12}}
	\\
=& \, \aiprd{ \muh^{m+\frac12}}{\phich^{m+\frac12}} - \varepsilon^{-1} \,\iprd{\phitil^{m+\frac12}}{\Delta_h \phich^{m+\frac12}}
	\\
&+ \varepsilon^{-1} \,\iprd{\chi\left(\phih^{m+1},\phih^m\right)}{\Delta_h \phich^{m+\frac12}}
	\\
\le& \, \frac12\norm{\nabla\muh^{m+\frac12}}{L^2}^2 +\frac12 \norm{\nabla\phich^{m+\frac12}}{L^2}^2  + C \norm{\phitil^{m+\frac12}}{L^2}^2 + \frac{\varepsilon}4 \norm{\Delta_h \phich^{m+\frac12}}{L^2}^2 
	\\
&+ C \norm{\chi\left(\phih^{m+1},\phih^m\right)}{L^2}^2 + \frac{\varepsilon}4 \norm{\Delta_h \phich^{m+\frac12}}{L^2}^2.
	\end{align*}
Using the triangle inequality, \eqref{eq:Linf-H1-phi-discrete}, and \eqref{eq:chi-L2-stability}, we have 
	\begin{equation}
\varepsilon \norm{ \Delta_h \phich^{m+\frac12}}{L^2}^2 \leq \norm{\nabla \muh^{m+\frac12}}{L^2}^2  + C.
	\label{eq:discrete-Laplacian-phi-middle}
	\end{equation}
Applying $\tau\sum_{m=1}^{M-1}$, estimate (\ref{eq:sum-3D-good-phicheck}.1) now follows from (\ref{eq:sum-mu-discrete}).

To prove estimate (\ref{eq:sum-3D-good-phicheck}.2), we use the discrete Gagliardo-Nirenberg inequality:
	\begin{equation}
\norm{\psi_h^m}{L^\infty} \leq C\norm{\Delta_h \psi_h^m}{L^2}^{\frac{d}{2(6-d)}} \,\norm{\psi_h^m}{L^6}^{\frac{3(4-d)}{2(6-d)}} + C\norm{\psi_h^m}{L^6} \quad \forall \, \psi \in S_h, \qquad (d=2,3) .
	\label{eq:infinity-bound}
	\end{equation}
Applying $\tau\sum_{m=1}^{M-1}$ and using $H^1(\Omega) \hookrightarrow L^6(\Omega)$, (\ref{eq:Linf-H1-phi-discrete}), and (\ref{eq:sum-3D-good-phicheck}.1), estimate (\ref{eq:sum-3D-good-phicheck}.2) follows.  
	\end{proof}

\begin{lem}
	\label{lem-a-priori-stability-dtau-mu}
Let $(\phih^{m+1}, \muh^{m+\frac12}) \in S_h\times S_h$ be the unique solution of  \eqref{eq:scheme-a} -- \eqref{eq:scheme-b}, and $(\phih^1, \muh^{\frac12}) \in S_h\times S_h$, the unique solution of \eqref{eq:scheme-a-initial} -- \eqref{eq:scheme-b-initial}. Assume that $\norm{\muh^{0}}{L^2}^2 \le C$, independent of $h$. Then the following estimates hold for any $h,\, \tau >0$:
	\begin{align}
\tau \sum_{m=0}^{M-1} \norm{\dtau\phih^{m+\frac12}}{L^2}^2 & \le C (T+1),
	\label{eq:sum-dtau-phi}
	\\
\max_{0\le m\le M-1} \norm{\muh^{m+\frac12}}{L^2}^2 &\le C (T+1),
	\label{eq:Linf-mu} 
	\end{align}
for some constant $C>0$ that is independent of $h$, $\tau$, and $T$.
	\end{lem}
	
	\begin{proof}
The proof is divided into three parts.

\noindent \\ \textbf{Part 1:} We first establish
	\begin{equation}
\norm{\muh^{\frac12}}{L^2}^2 + \tau \, \norm{\dtau \phih^{\frac12}}{L^2}^2 \le C.
	\end{equation}
To this end, setting $\nu = \tau \dtau\phih^{\hf}$ in \eqref{eq:scheme-a-initial} and $\psi = 2\muh^{\hf}$ in \eqref{eq:scheme-b-initial} and adding the resulting equations, we have  
	\begin{align*}
2\norm{\muh^{\frac12}}{L^2}^2  + \tau \norm{\dtau \phih^{\frac12}}{L^2}^2 =& \, \frac{2}{\varepsilon} \iprd{\chi\left(\phih^1,\phih^0\right)}{\muh^{\frac12}} -\frac{2}{\varepsilon}\iprd{\phih^0}{\muh^{\hf}} - \tau \iprd{\Delta_h\muh^0}{\muh^{\frac12}} - 2\varepsilon\iprd{\Delta_h\phih^0}{\muh^{\hf}}
	\\
\le& \,  \norm{\muh^{\frac12}}{L^2}^2 + C \norm{\chi\left(\phih^1,\phih^0\right)}{L^2}^2 
	\\
&+ C \norm{\phi_h^0}{L^2}^2 + C \, \tau^2 \norm{\Delta_h \muh^0}{L^2}^2 + C \norm{\Delta_h \phih^0}{L^2}^2 .
	\end{align*}
Thus, 
	\begin{align}
\norm{\muh^{\frac12}}{L^2}^2 + \tau \norm{\dtau \phih^{\frac12}}{L^2}^2 &\le C,
	\label{eq:mu-initial-control}
	\end{align}
considering the initial stability \eqref{eq:initial-stability}, \eqref{eq:Linf-H1-phi-discrete}, and \eqref{eq:chi-L2-stability}.
	
\noindent \textbf{Part 2:} Next we prove that
	\begin{equation}
 \norm{\muh^{\frac32}}{L^2}^2 + \tau \, \norm{\dtau \phih^{\frac32}}{L^2}^2 \le C .
	\end{equation}
Setting $m=1$ in \eqref{eq:scheme-b} and subtracting \eqref{eq:scheme-b-initial}, we obtain
	\begin{align}
\iprd{\muh^{\frac32} - \muh^{\frac12}}{\psi} =& \, \varepsilon \,\aiprd{\phich^{\frac32} - \phih^{\frac12}}{\psi} - \frac{3}{2 \varepsilon} \iprd{ \phih^1 - \phih^0}{\psi} - \frac{\tau}{2} \aiprd{\muh^0}{\psi}
	\nonumber
	\\
&+ \varepsilon^{-1} \iprd{\chi\left(\phih^2, \phih^1\right) - \chi\left(\phih^1, \phih^0\right)}{\psi} 
	\\
=& \varepsilon \, \aiprd{\frac34 \tau \dtau \phih^{\frac32} + \frac14 \tau \dtau \phih^{\frac12}}{\psi} - \frac{3}{2 \varepsilon} \iprd{ \phih^1 - \phih^0}{\psi} - \frac{\tau}{2} \aiprd{\muh^0}{\psi}
	\nonumber
	\\
&+ \varepsilon^{-1} \iprd{\chi\left(\phih^2, \phih^1\right) - \chi\left(\phih^1, \phih^0\right)}{\psi} .
	\label{eq:scheme-b-staggered-first-step}	
	\end{align}
Additionally, we take a weighted average of \eqref{eq:scheme-a} with $m=1$ and \eqref{eq:scheme-a-initial} with the weights $\frac34$ and $\frac14$, respectively, to obtain,
	\begin{align}
\iprd{\frac34 \dtau \phih^{\frac32} + \frac14 \dtau \phih^{\frac12}}{\nu} = - \varepsilon \, \aiprd{\frac34 \muh^{\frac32} + \frac14 \muh^{\frac12}}{\nu}, \qquad \forall \ \nu\in S_h.
	\label{eq:scheme-a-staggered-first-step}	
	\end{align}
Taking $\psi = \frac34 \muh^{\frac32} + \frac14 \muh^{\frac12}$ in \eqref{eq:scheme-b-staggered-first-step}, $\nu = \frac{3\tau}{4} \dtau \phih^{\frac32} + \frac{\tau}{4} \dtau \phih^{\frac12}$ in \eqref{eq:scheme-a-staggered-first-step}, and adding the results yields 
	\begin{align*}
\iprd{\muh^{\frac32} - \muh^{\frac12}}{\frac34 \muh^{\frac32} + \frac14 \muh^{\frac12}} + \tau \norm{\frac34 \dtau \phih^{\frac32} + \frac14 \dtau \phih^{\frac12}}{L^2}^2 =& \, - \frac{3}{8 \varepsilon} \iprd{\phih^1 - \phih^0}{3 \muh^{\frac32} +  \muh^{\frac12}} 
	\\
&- \frac{\tau}{8\varepsilon} \aiprd{\muh^0}{3 \muh^{\frac32} +  \muh^{\frac12}}
	\\
&+ \frac{1}{4 \varepsilon} \iprd{\chi\left(\phih^2, \phih^1\right) - \chi\left(\phih^1, \phih^0\right)}{3 \muh^{\frac32} + \muh^{\frac12}} 
	\\
=& \, \, - \frac{3}{8 \varepsilon} \iprd{\phih^1 - \phih^0}{3\muh^{\frac32} +  \muh^{\frac12}} 
	\\
&+ \frac{\tau}{8\varepsilon} \iprd{\Delta_h \muh^0}{3 \muh^{\frac32} +  \muh^{\frac12}}
	\\
&+ \frac{1}{4 \varepsilon} \iprd{\chi\left(\phih^2, \phih^1\right) - \chi\left(\phih^1, \phih^0\right)}{3 \muh^{\frac32} + \muh^{\frac12}} 
	\\
\le& \,\frac14 \norm{\muh^{\frac32}}{L^2}^2 + C \norm{\muh^{\frac12}}{L^2}^2 + C \norm{\phih^1}{L^2}^2 
	\\
&+ C \norm{\phih^0}{L^2}^2 + C \tau^2 \norm{\Delta_h \muh^0}{L^2}^2 
	\\
&+ C \norm{\chi\left(\phih^2, \phih^1\right)}{L^2}^2 + C \norm{\chi\left(\phih^1, \phih^0\right)}{L^2}^2 
	\\
\le& \,\frac14 \norm{\muh^{\frac32}}{L^2}^2 + C \norm{\muh^{\frac12}}{L^2}^2 + C,
	\end{align*}
where we have used Young's inequality, \eqref{eq:initial-stability}, \eqref{eq:Linf-H1-phi-discrete}, and \eqref{eq:chi-L2-stability}. Considering Part 1 and the inequalities
	\begin{align*}
\norm{\frac34 \dtau \phih^{\frac32} + \frac14 \dtau \phih^{\frac12}}{L^2}^2 =& \,  \frac9{16} \norm{\dtau \phih^{\frac32}}{L^2}^2 + \frac38 \iprd{\dtau \phih^{\frac32}}{\dtau \phih^{\frac12}} + \frac1{16} \norm{\dtau \phih^{\frac12}}{L^2}^2
	\\
\ge& \, \frac9{16} \norm{\dtau \phih^{\frac32}}{L^2}^2 - \frac38 \norm{\dtau \phih^{\frac32}}{L^2} \norm{\dtau \phih^{\frac12}}{L^2} + \frac1{16} \norm{\dtau \phih^{\frac12}}{L^2}^2
	\\
\ge& \, \frac38 \norm{\dtau \phih^{\frac32}}{L^2}^2 - \frac18 \norm{\dtau \phih^{\frac12}}{L^2}^2,
	\end{align*}
and
	\begin{align*}
\iprd{\muh^{\frac32} - \muh^{\frac12}}{\frac34 \muh^{\frac32} + \frac14 \muh^{\frac12}} =& \frac34 \norm{\muh^{\frac32}}{L^2}^2 - \frac12 \iprd{\muh^{\frac32}}{\muh^{\frac12}} - \frac14 \norm{\muh^{\frac12}}{L^2}^2
	\\
\ge& \frac12 \norm{\muh^{\frac32}}{L^2}^2 - \frac12 \norm{\muh^{\frac12}}{L^2}^2,
	\end{align*}
we have, 
	\begin{align}
\frac14 \norm{\muh^{\frac32}}{L^2}^2 + \frac{3\tau}{8} \norm{\dtau \phih^{\frac32}}{L^2}^2 \le C \norm{\muh^{\frac12}}{L^2}^2 + \frac{\tau}{8} \norm{\dtau \phih^{\frac12}}{L^2}^2 + C \le C.
	\label{eq:mu-control-second-step}
	\end{align}

\noindent \textbf{Part 3:} Finally, we will establish
	\begin{equation}
\norm{\muh^{\ell+\frac12}}{L^2}^2 + \frac{\tau}{8} \sum\limits_{m=2}^{\ell} \norm{\dtau \phih^{m+\frac12}}{L^2}^2 \le C (T+1).
	\end{equation}
For $2 \le m \le M-1$, we subtract \eqref{eq:scheme-b} from itself at consecutive time steps to obtain
	\begin{align}
\iprd{\muh^{m+\frac12} - \muh^{m-\frac12}}{\psi} =& \, \varepsilon \,\aiprd{\phich^{m+\frac12} - \phich^{m-\frac12}}{\psi} - \varepsilon^{-1} \iprd{\phitil^{m+\frac12} - \phitil^{m-\frac12}}{\psi}
	\nonumber
	\\
&+ \varepsilon^{-1} \iprd{\chi\left(\phih^{m+1}, \phih^m\right) - \chi\left(\phih^m, \phih^{m-1}\right)}{\psi} 
	\nonumber
	\\
=&\, \varepsilon \,\aiprd{\frac34 \tau \dtau \phih^{m+\frac12} + \frac14 \tau \dtau \phih^{m-\frac32}}{\psi} - \varepsilon^{-1} \iprd{\frac32 \tau \dtau \phih^{m-\frac12} - \frac12 \tau \dtau \phih^{m-\frac32}}{\psi}
	\nonumber
	\\
&+ \frac{1}{4 \varepsilon} \iprd{\omega^m_h \left(\phih^{m+1} - \phih^{m-1}\right)}{\psi},
	\label{eq:scheme-b-staggered}
	\end{align}
for all $\psi \in S_h$, where  $\omega^m_h := \omega\left(\phih^{m+1}, \phih^m, \phih^{m-1}\right)$ and
	\begin{align*}
\omega\left(a,b,c\right) := a^2+b^2+c^2 +ab +bc + ac.
	\end{align*}
Additionally, we take a weighted average of the $m+\frac12$ and $m-\frac32$ time steps with the weights $\frac34$ and $\frac14$, respectively, of \eqref{eq:scheme-a} to obtain,
	\begin{equation}
\iprd{\frac34 \dtau \phih^{m+\frac12} + \frac14 \dtau \phih^{m-\frac32}}{\nu} = - \varepsilon \, \aiprd{\frac34 \muh^{m+\frac12} + \frac14 \muh^{m-\frac32}}{\nu}, 
	\label{eq:scheme-a-staggered}
	\end{equation}
for all $\nu \in S_h$, which is well-defined for all $2 \le m \le M-1$. Taking $\psi = \frac34 \muh^{m+\frac12} + \frac14 \muh^{m-\frac32}$ in \eqref{eq:scheme-b-staggered}, $\nu = \tau \left(\frac34 \dtau \phih^{m+\frac12} + \frac14 \dtau \phih^{m-\frac32} \right)$ in \eqref{eq:scheme-a-staggered}, and adding the results yields 
	\begin{align*}
\iprd{\muh^{m+\frac12} - \muh^{m-\frac12}}{\frac34 \muh^{m+\frac12} + \, \frac14 \muh^{m-\frac32}} +& \, \tau \norm{\frac34 \dtau \phih^{m+\frac12} + \frac14 \dtau \phih^{m-\frac32}}{L^2}^2 
	\\
=& \, -\frac{\tau}{\varepsilon} \iprd{\frac32 \dtau \phih^{m-\frac12} - \frac12 \dtau \phih^{m-\frac32}}{\frac34 \muh^{m+\frac12} + \frac14 \muh^{m-\frac32}}
	\\
&+ \frac{1}{4\varepsilon} \iprd{\omega^m_h \left(\phih^{m+1} - \phih^{m-1}\right) }{\frac34 \muh^{m+\frac12} + \frac14 \muh^{m-\frac32}}
	\\
=& \, -\frac{\tau}{\varepsilon} \iprd{\frac32 \dtau \phih^{m-\frac12} - \frac12 \dtau \phih^{m-\frac32}}{\frac34 \muh^{m+\frac12} + \frac14 \muh^{m-\frac32}}
	\\
&+ \frac{\tau}{4 \varepsilon} \iprd{\omega^m_h \dtau \phih^{m+\frac12} }{\frac34 \muh^{m+\frac12} + \frac14 \muh^{m-\frac32}}
	\\
&+ \frac{\tau}{4 \varepsilon} \iprd{\omega^m_h \dtau \phih^{m-\frac12}}{\frac34 \muh^{m+\frac12} + \frac14 \muh^{m-\frac32}}
	\\
\le& \, \frac{3 \tau}{8 \varepsilon} \norm{\dtau \phih^{m-\frac12}}{L^2} \norm{3\muh^{m+\frac12} + \muh^{m-\frac32}}{L^2}
	\\
&+ \frac{\tau}{8 \varepsilon} \norm{\dtau \phih^{m-\frac32}}{L^2} \norm{3 \muh^{m+\frac12} + \muh^{m-\frac32}}{L^2}
	\\
&+ \frac{\tau}{16 \varepsilon} \norm{\omega^m_h}{L^3} \norm{\dtau \phih^{m+\frac12}}{L^2} \norm{3\muh^{m+\frac12} + \muh^{m-\frac32}}{L^6}
	\\
&+ \frac{\tau}{16\varepsilon} \norm{\omega^m_h}{L^3} \norm{\dtau \phih^{m-\frac12}}{L^2} \norm{3\muh^{m+\frac12} + \muh^{m-\frac32}}{L^6}.
	\end{align*}
Hence,
\begin{align*}
\iprd{\muh^{m+\frac12} - \muh^{m-\frac12}}{\frac34 \muh^{m+\frac12} + \frac14 \muh^{m-\frac32}} +& \, \tau \norm{\frac34 \dtau \phih^{m+\frac12} + \frac14 \dtau \phih^{m-\frac32}}{L^2}^2 
	\\
\le& \, \frac{\tau}{8} \norm{\dtau \phih^{m+\frac12}}{L^2}^2 + \frac{\tau}{32} \norm{\dtau \phih^{m-\frac12}}{L^2}^2 + \frac{\tau}{32} \norm{\dtau \phih^{m-\frac32}}{L^2}^2 
	\\
&+ C \tau \norm{\muh^{m+\frac12}}{H^1}^2 + C \tau \norm{\muh^{m-\frac32}}{H^1}^2 ,
	\end{align*}
where we use the $H^1(\Omega) \hookrightarrow L^6(\Omega)$ embedding to achieve following bound,  
	\begin{align*}
\norm{\omega^m_h}{L^3} =& \, \norm{\left(\phih^{m+1}\right)^2 + \left(\phih^{m}\right)^2 + \left(\phih^{m-1}\right)^2 + \phih^{m+1} \phih^m + \phih^{m+1} \phih^{m-1} + \phih^m \phih^{m-1}}{L^3}
	\\
\le& \, C \norm{\phih^{m+1}}{L^6}^2  + C \norm{\phih^{m}}{L^6}^2 + C \norm{\phih^{m-1}}{L^6}^2 \le C.
	\end{align*}
Applying $\sum_{m=2}^{\ell}$ and using the following properties
	\begin{align*}
\iprd{\muh^{m+\frac12} - \muh^{m-\frac12}}{\frac34 \muh^{m+\frac12} + \frac14 \muh^{m-\frac32}} =& \, \frac12 \iprd{\muh^{m+\frac12} - \muh^{m-\frac12}}{\muh^{m+\frac12} + \muh^{m-\frac12}}
	\\
&+ \frac14 \iprd{\muh^{m+\frac12} - \muh^{m-\frac12}}{\muh^{m+\frac12} - 2\muh^{m-\frac12} + \muh^{m-\frac32}}
	\\
=& \, \frac12 \norm{\muh^{m+\frac12}}{L^2}^2 - \frac12 \norm{\muh^{m-\frac12}}{L^2}^2 + \frac18 \norm{\muh^{m+\frac12} - \muh^{m-\frac12}}{L^2}^2 
	\\
&- \frac18 \norm{\muh^{m-\frac12} - \muh^{m-\frac32}}{L^2}^2 + \frac18 \norm{\muh^{m+\frac12} - 2 \muh^{m-\frac12} + \muh^{m-\frac32}}{L^2}^2,
	\\
\norm{\frac34 \dtau \phih^{m+\frac12} + \frac14 \dtau \phih^{m-\frac32}}{L^2}^2 =& \, \frac9{16} \norm{\dtau \phih^{m+\frac12}}{L^2}^2 + \frac38 \iprd{\dtau \phih^{m+\frac12}}{\dtau \phih^{m-\frac32}} + \frac1{16} \norm{\dtau \phih^{m-\frac32}}{L^2}^2
	\\
\ge& \, \frac9{16} \norm{\dtau \phih^{m+\frac12}}{L^2}^2 - \frac38 \norm{\dtau \phih^{m+\frac12}}{L^2} \norm{\dtau \phih^{m-\frac32}}{L^2} 
	\\
&+ \frac1{16} \norm{\dtau \phih^{m-\frac32}}{L^2}^2
	\\
\ge& \, \frac38 \norm{\dtau \phih^{m+\frac12}}{L^2}^2 - \frac18 \norm{\dtau \phih^{m-\frac32}}{L^2}^2,
	\end{align*}
we conclude
	\begin{align*}
\frac12 \norm{\muh^{\ell+\frac12}}{L^2}^2 + \frac{\tau}{16} \sum_{m=2}^{\ell} \norm{\dtau \phih^{m+\frac12}}{L^2}^2 \le& \, \frac18 \norm{\muh^{\frac32} - \muh^{\frac12}}{L^2}^2 + \frac{3 \tau}{16} \norm{\dtau \phih^{\frac32}}{L^2}^2 + \frac12 \norm{\muh^{\frac32}}{L^2}^2
	\\
&+ \frac{5 \tau}{32} \norm{\dtau \phih^{\frac12}}{L^2}^2  + C \tau \sum_{m=0}^{\ell}\norm{\muh^{m+\frac12}}{H^1}^2  \le C (T+1),
	\end{align*}
for any $2 \le \ell \le M-1$, where we have used Parts 1 and 2 and estimates  \eqref{eq:sum-mu-discrete} and \eqref{eq:sum-3D-good-mu}. The proof is completed by combining all three parts.
	\end{proof}
	
	\begin{lem}
	\label{lem-a-priori-stability-phicheck}
Let $(\phih^{m+1}, \muh^{m+\frac12}) \in S_h\times S_h$ be the unique solution of  \eqref{eq:scheme-a} -- \eqref{eq:scheme-b}, and $(\phih^1, \muh^{\frac12}) \in S_h\times S_h$, the unique solution of \eqref{eq:scheme-a-initial} -- \eqref{eq:scheme-b-initial}. Then the following estimates hold for any $h,\, \tau >0$:
	\begin{align}
\norm{\Delta_h \phih^{\frac12}}{L^2}^2  + \norm{\phih^{\frac12}}{L^\infty}^2 &\le C,
	\label{eq:Linf-phih-ave}
	\\
\max_{1\le m\le M-1} \bigg[ \norm{\Delta_h \phich^{m+\frac12}}{L^2}^2  + \norm{\phich^{m+\frac12}}{L^\infty}^{\frac{4(6-d)}{d}} \bigg] &\le C(T+1),
	\label{eq:Linf-phich} 
	\end{align}
for some constant $C>0$ that is independent of $h$, $\tau$, and $T$.
	\end{lem}
	
	\begin{proof}
To prove (\ref{eq:Linf-phih-ave}.1), set $\psi = \Delta_h \phih^{\frac12}$ in \eqref{eq:scheme-b-initial} and use the definition of the discrete Laplacian \eqref{eq:discrete-Laplacian} to obtain
	\begin{align*}
\varepsilon \norm{\Delta_h \phih^{\frac12}}{L^2}^2 =& \, -\varepsilon \, \aiprd{\phih^{\frac12}}{\Delta_h \phih^{\frac12}}
	\\
=& \, \varepsilon^{-1} \iprd{\chi\left(\phih^1,\phih^0\right)}{\Delta_h \phih^{\frac12}} - \varepsilon^{-1} \iprd{\phih^0}{\Delta_h \phih^{\frac12}} - \iprd{\muh^{\frac12}}{\Delta_h \phih^{\frac12}} + \frac{\tau}{2} \aiprd{\muh^0}{\Delta_h \phih^{\frac12}}
	\\
\le& \, \frac{\varepsilon}{2} \norm{\Delta_h \phih^{\frac12}}{L^2}^2 + C \norm{\chi\left(\phih^1,\phih^0\right)}{L^2}^2 + C \norm{\phih^0}{L^2}^2 + C \norm{\muh^{\frac12}}{L^2}^2 + C \tau^2 \norm{\Delta_h \muh^0}{L^2}^2
	\\
\le& \, \frac{\varepsilon}{2} \norm{\Delta_h \phih^{\frac12}}{L^2}^2 + C .
	\end{align*}
The result now follows. Estimate (\ref{eq:Linf-phih-ave}.2) follows from \eqref{eq:infinity-bound}, the embedding $H^1(\Omega) \hookrightarrow L^6(\Omega)$, \eqref{eq:Linf-H1-phi-discrete}, and (\ref{eq:Linf-phih-ave}.1). 

Setting $\psi = \Delta_h \phich^{m+\frac12}$ in \eqref{eq:scheme-b} and using the definition of the discrete Laplacian \eqref{eq:discrete-Laplacian}, we get
	\begin{align*}
\varepsilon \norm{\Delta_h \phich^{m+\frac12}}{L^2}^2 =& \, -\varepsilon \, \aiprd{\phich^{m+\frac12}}{\Delta_h \phich^{m+\frac12}}
	\\
=& \, - \iprd{\muh^{m+\frac12}}{\Delta_h \phich^{m+\frac12}} - \varepsilon^{-1} \iprd{\phitil^{m+\frac12}}{\Delta_h \phich^{m+\frac12}}
	\\
&+ \varepsilon^{-1} \iprd{\chi\left(\phih^{m+1},\phih^m\right)}{\Delta_h \phich^{m+\frac12}}
	\\
\le& \, C \norm{\muh^{m+\frac12}}{L^2}^2 + \frac{\varepsilon}{2} \norm{\Delta_h \phich^{m+\frac12}}{L^2}^2 + C \norm{\phitil^{m+\frac12}}{L^2}^2
	\\
&  + C \norm{\chi\left(\phih^{m+1},\phih^m\right)}{L^2}^2 
	\\
\le& \, C + C \norm{\muh^{m+\frac12}}{L^2}^2 + \frac{\varepsilon}{2} \norm{\Delta_h \phich^{m+\frac12}}{L^2}^2,
	\end{align*}
where we have used the triangle inequality and \eqref{eq:chi-L2-stability}. Hence, $\norm{\Delta_h \phich^{m+\frac12}}{L^2}^2 \le C + C\norm{\muh^{m+\frac12}}{L^2}^2$, for $1\le m \le M-1$, and estimate (\ref{eq:Linf-phich}.1) follows from \eqref{eq:Linf-mu}. Estimate (\ref{eq:Linf-phich}.2) follows from \eqref{eq:infinity-bound}, the embedding $H^1(\Omega) \hookrightarrow L^6(\Omega)$, \eqref{eq:Linf-H1-phi-discrete}, and (\ref{eq:Linf-phich}.1). 
	\end{proof}
	
	\begin{lem}
	\label{lem-a-priori-stability-ultimate}
Let $(\phih^{m+1}, \muh^{m+\frac12}) \in S_h\times S_h$ be the unique solution of  \eqref{eq:scheme-a} -- \eqref{eq:scheme-b}, and $(\phih^1, \muh^{\frac12}) \in S_h\times S_h$, the unique solution of \eqref{eq:scheme-a-initial} -- \eqref{eq:scheme-b-initial}. The following estimates hold for any $h,\, \tau >0$:
	\begin{align}
\max_{0\le m\le M} \bigg[ \norm{\Delta_h \phih^m}{L^2}^2  + \norm{\phih^m}{L^\infty}^{\frac{4(6-d)}{d}} \bigg] &\le C(T+1),
	\label{eq:Linf-phi} 
	\end{align}
for some constant $C>0$ that is independent of $h$, $\tau$, and $T$.
	\end{lem}

	\begin{proof}
We begin by proving the stability for the first time step. A simple application of the triangle inequality gives (\ref{eq:Linf-phi}.1) for $m=1$ as follows,
	\begin{align*}
\norm{\Delta_h \phih^1}{L^2} =  \norm{\Delta_h \phih^1 + \Delta_h \phih^0 - \Delta_h \phih^0}{L^2} &\le \norm{\Delta_h \phih^1 + \Delta_h \phih^0}{L^2} + \norm{\Delta_h \phih^0}{L^2}
	\\
&\le 2 \norm{\Delta_h \phih^{\frac12}}{L^2} + \norm{\Delta_h \phih^0}{L^2} \le C,
	\end{align*}
where we have used the stability of the initial data, inequality \eqref{eq:initial-stability}, and (\ref{eq:Linf-phih-ave}.1).  Next, using \eqref{eq:infinity-bound}, $H^1(\Omega) \hookrightarrow L^6(\Omega)$, \eqref{eq:Linf-H1-phi-discrete}, and  (\ref{eq:Linf-phi}.1), we arrive at (\ref{eq:Linf-phi}.2) for $m=1$. For $2 \le m \le M-1$, by definition,
	\begin{align*}
\norm{\Delta_h \phich^{m+\frac12}}{L^2}^2 &= \norm{\Delta_h \left(\frac34 \phih^{m+1} + \frac14 \phih^{m-1}\right)}{L^2}^2
	\\
&= \frac{9}{16} \norm{\Delta_h \phih^{m+1}}{L^2}^2 + \frac38 \iprd{\Delta_h \phih^{m+1}}{\Delta_h \phih^{m-1}} + \frac1{16} \norm{\Delta_h \phih^{m-1}}{L^2}^2
	\\
&\ge \frac{9}{16} \norm{\Delta_h \phih^{m+1}}{L^2}^2 - \frac{3}{16} \norm{\Delta_h \phih^{m+1}}{L^2}^2 - \frac{3}{16} \norm{\Delta_h \phih^{m-1}}{L^2}^2 + \frac{1}{16} \norm{\Delta_h \phih^{m-1}}{L^2}^2
	\\
&= \frac{3}{8} \norm{\Delta_h \phih^{m+1}}{L^2}^2 - \frac{1}{8} \norm{\Delta_h \phih^{m-1}}{L^2}^2.
	\end{align*}
Using induction and estimate (\ref{eq:Linf-phich}.1), we find
	\begin{align*}
\norm{\Delta_h \phih^{2m}}{L^2}^2 &\le \frac{8}{3}\left(1+\frac13+\left(\frac13\right)^2+\cdots+\left(\frac13\right)^{m-1}\right) C(T+1) + \left(\frac13\right)^m \norm{\Delta_h \phih^0}{L^2}^2
	\\
&\le \frac83 \cdot \frac32 C(T+1) + \left(\frac13\right)^m \cdot C \le C(T+1),
	\end{align*}
and 
	\begin{align*}
\norm{\Delta_h \phih^{2m+1}}{L^2}^2 &\le \frac83 \left(1 + \frac13 + \left(\frac13\right)^2 + \cdots + \left(\frac13\right)^{m-1}\right) C(T+1) + \left(\frac13\right)^m \norm{\Delta_h \phih^1}{L^2}^2
	\\
&\le \frac83 \cdot \frac32 C(T+1) + \left(\frac13\right)^m\cdot C \le C(T+1),
	\end{align*}
and estimate (\ref{eq:Linf-phi}.1) follows. Estimate (\ref{eq:Linf-phi}.2) follows from \eqref{eq:infinity-bound}, (\ref{eq:Linf-phi}.1), and the embedding $H^1(\Omega) \hookrightarrow L^6(\Omega)$.
	\end{proof}
	
	\section{Error Estimates for the Fully Discrete Convex Splitting Scheme}
	\label{sec:error estimates}

In this section, we provide a rigorous convergence analysis for our scheme in the appropriate energy norms. We shall assume that weak solutions have the additional regularities
	\begin{align}
	\nonumber
&\phi \in L^{\infty}\left(0,T;W^{1,6}(\Omega)\right) \cap H^1\left(0,T;H^{q+1}(\Omega)\right) \cap H^2(0,T;H^{3}(\Omega)) \cap H^3(0,T;L^2(\Omega)),
	\\
&\phi^2 \in H^2\left(0,T; H^1(\Omega)\right) ,
	\label{eq:higher-regularities}
	\\
&\mu \in L^2\left(0,T;H^{q+1}(\Omega)\right),
	\nonumber
	\end{align}
where $q\ge 1$.  The norm bounds associated to the assumed regularities above are not necessarily global-in-time and therefore can involve constants that depend upon the final time $T$.  We also assume that the initial data are sufficiently regular so that the stability \eqref{eq:initial-stability} holds.  Weak solutions $(\phi, \mu)$ to \eqref{eq:weak-mch-a} - \eqref{eq:weak-mch-b} with the higher regularities \eqref{eq:higher-regularities} solve the following variational problem: for all $t\in [0,T]$, 
	\begin{subequations}
	\begin{align}
\iprd{\partial_t \phi}{\nu} + \varepsilon \, \aiprd{\mu}{\nu} &= 0 &\forall \quad \nu \in H^1(\Omega),
	\\
\iprd{\mu}{\psi} - \varepsilon \, \aiprd{\phi}{\psi} - \varepsilon^{-1} \iprd{\phi^3 - \phi}{\psi} &=0 &\forall \quad \psi \in H^1(\Omega).
	\end{align}
	\end{subequations}
We define the following: for any \emph{real} number $m\in [0,M]$,
	\begin{displaymath}
t_{m} := m \, \tau, \quad \phi^{m} := \phi(t_{m}), \quad  \eA^{\phi,m} := \phi^{m} - R_h \phi^{m}, \quad \eA^{\mu,m} := \mu^{m} - R_h \mu^{m};
	\end{displaymath}
and for any integer $0\le m\le M-1$,
	\begin{eqnarray*}
\dtau \phi^{m+\frac12} := \frac{\phi^{m+1} - \phi^{m}}{\tau}, \quad \sigma^{m+\frac12}_1 := \dtau R_h \phi^{m+\frac12} - \dtau \phi^{m+\frac12},  
	\\
\sigma^{m+\frac12}_2 := \dtau \phi^{m+\frac12} - \partial_t \phi^{m+\frac12}, \quad \sigma^{m+\frac12}_3 := \frac12  \phi^{m+1} + \frac12  \phi^{m} -  \phi^{m+\frac12}
	\\
\sigma_4^{m+\hf} :=  \chi\left(\phi^{m+1},\phi^m \right) - \left(\phi^{m+\hf}\right)^3.
	\end{eqnarray*}	  
Then the PDE solution, evaluated at the half-integer time steps $t_{m+\hf}$, satisfies
	\begin{subequations}
	\begin{align}
\iprd{\dtau R_h \phi^{m+\frac12}}{\nu} + \varepsilon \, \aiprd{R_h \mu^{m+\frac12}}{\nu} =& \, \iprd{\sigma^{m+\frac12}_1 + \sigma^{m+\frac12}_2}{\nu} ,
	\label{eq:weak-error-a}
	\\
\varepsilon \, \aiprd{\frac12 R_h \phi^{m+1} + \frac12 R_h \phi^{m}}{\psi} - \iprd{R_h \mu^{m+\frac12}}{\psi} =& \, \iprd{\eA^{\mu,m+\frac12}}{\psi} - \frac{1}{\varepsilon} \iprd{\chi\left(\phi^{m+1},\phi^m\right)}{\psi} 
	\nonumber
	\\
&+  \frac{1}{\varepsilon} \iprd{\phi^{m+\frac12}}{\psi} + \varepsilon \, \aiprd{\sigma^{m+\frac12}_3}{\psi}
	\nonumber
	\\
& +\frac{1}{\varepsilon}\iprd{\sigma_4^{m+\hf}}{\psi} 
	\label{eq:weak-error-b}
	\end{align}
	\end{subequations}
for all $\nu, \psi \in S_h$.  Restating the fully discrete splitting scheme, Eqs.~\eqref{eq:scheme-a} -- \eqref{eq:scheme-b} and \eqref{eq:scheme-a-initial} -- \eqref{eq:scheme-b-initial}, we have, for all $\nu, \psi \in S_h$,
	\begin{subequations}
	\begin{align}
\iprd{\dtau \phih^{\frac12}}{\nu} + \varepsilon \,\aiprd{\muh^{\frac12}}{\nu} &= \, 0, 
	\label{eq:scheme-error-a-initial}
	\\
\varepsilon \,\aiprd{\phih^{\frac12}}{\psi} - \iprd{\muh^{\frac12}}{\psi} &= \, - \frac{1}{\varepsilon} \iprd{\chi\left(\phih^1,\phih^0\right)}{\psi} + \frac{1}{\varepsilon} \iprd{\phih^0 +\frac{\tau}{2} \partial_t \phi^0}{\psi} ;
	\label{eq:scheme-error-b-initial}
	\end{align}
	\end{subequations}
and, for $1 \le m \le M-1$, and all $\nu, \psi \in S_h$, 
	\begin{subequations}
	\begin{align}
\iprd{\dtau \phih^{m+\frac12}}{\nu} + \varepsilon \, \aiprd{\muh^{m+\frac12}}{\nu} =& \, 0 ,
	\label{eq:scheme-error-a}
	\\
	\nonumber
\varepsilon \, \aiprd{\phih^{m+\frac12}}{\psi} + \frac{\varepsilon}{4} \, \aiprd{\phih^{m+1} -2 \phih^m + \phih^{m-1}}{\psi} - \iprd{\muh^{m+\frac12}}{\psi} = &- \frac{1}{\varepsilon} \iprd{\chi\left(\phih^{m+1}, \phih^m\right)}{\psi}
	\nonumber
	\\
& + \frac{1}{\varepsilon} \iprd{\tilde{\phih}^{m+\hf}}{\psi}.
	\label{eq:scheme-error-b}
	\end{align}
	\end{subequations}
Now let us define the following additional error terms: for any integers $0\le m\le M$, 
	\begin{equation}
\eh^{\phi,m} := R_h \phi^{m} - \phih^{m}, \quad \e^{\phi,m} := \phi^{m} - \phih^{m}, 
	\end{equation}
and, for any integers $0\le m\le M-1$
	\begin{equation}
 \eh^{\mu,m+\hf} := R_h \mu^{m+\hf} - \muh^{m+\hf}, \quad  \e^{\mu,m+\hf} :=  \mu^{m+\hf} - \muh^{m+\hf}.
	\end{equation}
Setting $m=0$ in \eqref{eq:weak-error-a} -- \eqref{eq:weak-error-b} and subtracting \eqref{eq:scheme-error-a-initial} -- \eqref{eq:scheme-error-b-initial}, we have
	\begin{subequations}
	\begin{align}
\iprd{\dtau \eh^{\phi,\frac12}}{\nu} + \varepsilon \, \aiprd{\eh^{\mu,\frac12}}{\nu} =& \iprd{\sigma^{\frac12}_1 + \sigma^{\frac12}_2}{\nu} ,
	\label{eq:error-a-initial}
	\\
\frac{\varepsilon}{2} \, \aiprd{\eh^{\phi,1} + \eh^{\phi,0}}{\psi} - \iprd{\eh^{\mu,\frac12}}{\psi} =& \, \iprd{\eA^{\mu,\frac12}}{\psi} - \frac{1}{\varepsilon}\iprd{ \chi\left(\phi^1,\phi^0\right)- \chi\left(\phih^1,\phih^0\right)}{\psi} 
	\nonumber
	\\
& + \frac{1}{\varepsilon} \iprd{\phi^{\frac12} - \phi_h^0 - \frac{\tau}{2}\partial_t \phi^0}{\psi} + \varepsilon \, \aiprd{\sigma_3^{\hf}}{\psi}
	\nonumber
	\\
& +\frac{1}{\varepsilon}\iprd{\sigma_4^{\hf}}{\psi} .
	\label{eq:error-b-initial}
	\end{align}
	\end{subequations}
Similarly, subtracting \eqref{eq:scheme-error-a} -- \eqref{eq:scheme-error-b} from \eqref{eq:weak-error-a} -- \eqref{eq:weak-error-b}, yields, for $1 \le m \le M-1$,
	\begin{subequations}
	\begin{align}
\iprd{\dtau \eh^{\phi,m+\frac12}}{\nu} + \varepsilon \, \aiprd{\eh^{\mu,m+\frac12}}{\nu}  &= \iprd{\sigma^{m+\frac12}_1 + \sigma^{m+\frac12}_2}{\nu} ,
	\label{eq:error-a}
	\\
& \hspace{-2.2in}\frac{\varepsilon}{2} \, \aiprd{\eh^{\phi,m+1} + \eh^{\phi,m}}{\psi} + \frac{\varepsilon\tau^2}{4} \aiprd{\ddtau \eh^{\phi,m}}{\psi} - \iprd{\eh^{\mu,m+\frac12}}{\psi} 
	\nonumber
	\\
& = \iprd{\eA^{\mu,m+\frac12}}{\psi} - \frac{1}{\varepsilon} \iprd{\chi\left(\phi^{m+1},\phi^m\right) - \chi\left(\phih^{m+1},\phih^m\right)}{\psi} 
	\nonumber
	\\
&  + \frac{1}{\varepsilon} \iprd{ \phi^{m+\frac12}-\tilde{\phih}^{m+\frac12}}{\psi}+ \varepsilon \, \aiprd{\sigma_3^{m+\hf}}{\psi}
	\nonumber
	\\
& +\frac{1}{\varepsilon}\iprd{\sigma_4^{m+\hf}}{\psi} + \frac{\varepsilon\tau^2}{4} \, \aiprd{\ddtau \phi^m}{\psi},
	\label{eq:error-b}
	\end{align}
	\end{subequations}
where $\tau^2 \ddtau \psi^m := \psi^{m+1} - 2 \psi^m + \psi^{m-1}$.

Now, define the additional error terms
	\begin{align}
\sigma_5^{m+\hf} &:= \, \chi\left(\phih^{m+1},\phih^m \right) - \chi\left(\phi^{m+1},\phi^m \right)  ,
	\\
\sigma_6^{m+\hf}&:= \, \phi^{m+\hf} - \left\{
	\begin{array}{rcl}
\phih^0 + \frac{\tau}{2}\partial_t\phi^0, & \mbox{for} & m =0
	\\
\tilde{\phih}^{m+\hf}, & \mbox{for} & 1\le m \le M-1
	\end{array}
\right. .
	\end{align}
Then, setting $\nu = \eh^{\mu,\frac12}$ in \eqref{eq:error-a-initial} and $\psi = \dtau \eh^{\phi,\frac12}$ in \eqref{eq:error-b-initial}, setting $\nu = \eh^{\mu, m+\frac12}$ in \eqref{eq:error-a} and $\psi = \dtau \eh^{\phi, m+\frac12}$ in \eqref{eq:error-b}, and adding the resulting equations, we have
	\begin{align}
& \hspace{-0.75in} \frac{\varepsilon}{2} \aiprd{\eh^{\phi,m+1} + \eh^{\phi,m}}{\dtau \eh^{\phi, m+\frac12}} + \frac{\gamma_m\varepsilon\tau^2}{4} \, \aiprd{\ddtau\eh^{\phi,m}}{\dtau \eh^{\phi, m+\frac12}} + \varepsilon \norm{\nabla \eh^{\mu,m+\frac12}}{L^2}^2 
	\nonumber
	\\
 = & \iprd{\sigma^{m+\frac12}_1 + \sigma^{m+\frac12}_2}{\eh^{\mu,m+\frac12}} + \iprd{\eA^{\mu,m+\frac12}}{\dtau \eh^{\phi,m+\frac12}} + \varepsilon \, \aiprd{\sigma_3^{m+\hf}}{\dtau \eh^{\phi, m+\frac12}} 
	\nonumber
	\\
& + \frac{1}{\varepsilon} \iprd{ \sigma_4^{m+\hf} + \sigma_5^{m+\hf} + \sigma_6^{m+\hf} }{\dtau \eh^{\phi, m+\frac12}}  + \frac{\gamma_m\varepsilon\tau^2}{4} \, \aiprd{\ddtau\phi^m }{\dtau \eh^{\phi, m+\frac12}} ,
	\label{eq:error-eq}
	\end{align}
for all $0 \le m \le M-1$, where $\gamma_m := 1-\delta_{0,m}$ and $\delta_{k,\ell}$ is the Kronecker delta function.  The terms involving $\gamma_m$ are ``turned on" only when $m \ge 1$. Expression~\eqref{eq:error-eq} is the key error equation from which we will define our error estimates.

	\begin{lem}
	\label{lem-truncation-errors}
Suppose that $(\phi, \mu)$ is a weak solution to \eqref{eq:weak-error-a} -- \eqref{eq:weak-error-b}, with the additional regularities \eqref{eq:higher-regularities}.  Then for all $t_m\in[0,T]$ and for any $h$, $\tau >0$, there exists a constant $C>0$, independent of $h$ and $\tau$ and $T$, such that
	\begin{align}
\norm{\sigma^{m+\frac12}_1}{L^2}^2 &\le  C\frac{h^{2q+2}}{\tau} \int_{t_m}^{t_{m+1}} \norm{\partial_s\phi(s)}{H^{q+1}}^2  ds , \hspace{-0.085in}   & 0\le m\le M-1,
	\label{eq:truncation-1} 
	\\
\norm{\sigma^{m+\frac12}_2}{L^2}^2 &\le \frac{\tau^3}{640} \int_{t_m}^{t_{m+1}}\norm{\partial_{sss}\phi(s)}{L^2}^2 ds,   & 0\le m\le M-1,
	\label{eq:truncation-2}
	\\
\norm{\nabla \Delta \sigma^{m+\frac12}_3}{L^2}^2 &\le \frac{\tau^3}{96} \int_{t_m}^{t_{m+1}} \norm{\nabla \Delta \partial_{ss} \phi(s)}{L^2}^2 ds,   & 0\le m\le M-1,
	\label{eq:truncation-3}
	\\
\norm{\nabla \sigma^{m+\frac12}_3}{L^2}^2 &\le \frac{\tau^3}{96} \int_{t_m}^{t_{m+1}} \norm{\nabla \partial_{ss} \phi(s)}{L^2}^2 ds,   & 0\le m\le M-1,
	\label{eq:truncation-4}
	\\
\norm{ \hf\left(\phi^{m+1}\right)^2+\hf\left(\phi^m\right)^2 - \left(\phi^{m+\hf}\right)^2}{H^1}^2 &\le \frac{\tau^3}{96} \int_{t_m}^{t_{m+1}} \norm{\partial_{ss} \phi^2(s)}{H^1}^2 ds,  & 0\le m\le M-1,
	\label{eq:truncation-4-b}
	\\
\norm{\tau^2\nabla \Delta \ddtau\phi^m}{L^2}^2 &\le \frac{\tau^3}{3} \int_{t_{m-1}}^{t_{m+1}} \norm{\nabla \Delta \partial_{ss} \phi(s)}{L^2}^2 ds,  & 1\le m\le M-1,
	\label{eq:truncation-5}
	\\
\norm{\tau^2\nabla \ddtau\phi^m}{L^2}^2 &\le \frac{\tau^3}{3} \int_{t_{m-1}}^{t_{m+1}} \norm{\nabla \partial_{ss} \phi(s)}{L^2}^2 ds,   & 1\le m\le M-1,
	\label{eq:truncation-6}
	\\
\norm{\nabla\left(\phi^{m+\hf}- \frac32 \phi^m +\frac12\phi^{m-1}\right)}{L^2}^2 &\le  \frac{\tau^3}{12}  \int_{t_{m-1}}^{t_{m+1}}\norm{\nabla\partial_{ss}\phi(s)}{L^2}^2 \, ds,   & 1\le m\le M-1 ,
	\label{eq:truncation-9}
	\\
\norm{\nabla\left(\phi^{\hf}-\phi^0 -\frac{\tau}{2}\partial_t\phi^0\right) }{L^2} &\le  \frac{\tau^3}{24} \int_{t_0}^{t_\hf} \norm{\nabla \partial_{ss} \phi(s)}{L^2}^2 ds  . &
	\label{eq:truncation-10}
	\end{align}
	\end{lem}
	
	\begin{proof}
The proof of each of the inequalities above is a direct application of Taylor's Theorem with integral remainder. We suppress the details for the sake of brevity.
	\end{proof}
	
	\begin{lem}
	\label{lem-sigma4-estimate}
Suppose that $(\phi, \mu)$ is a weak solution to \eqref{eq:weak-error-a} -- \eqref{eq:weak-error-b}, with the additional regularities \eqref{eq:higher-regularities}. Then, there exists a constant $C>0$ independent of $h$ and $\tau$ -- but possibly dependent upon $T$ through the regularity estimates -- such that, for any $h, \tau >0$,
	\begin{align}
\norm{\nabla \sigma_4^{m+\hf}}{L^2}^2 \le& C\tau^3 \int_{t_m}^{t_{m+1}} \norm{\nabla \partial_{ss} \phi(s)}{L^2}^2 ds +C \tau^3 \int_{t_m}^{t_{m+1}} \norm{\partial_{ss} \phi^2(s)}{H^1}^2 ds.
	\end{align}
	\end{lem}
	\begin{proof}
We begin with the expansion
	\begin{align}
\nabla\sigma_4^{m+\hf} =&\, \left(\hf\phi^{m+1} +\hf \phi^m - \phi^{m+\hf} \right) \nabla\left(\hf\left(\phi^{m+1}\right)^2+\hf\left(\phi^m\right)^2 \right)
	\nonumber
	\\
& + \left(\hf\left(\phi^{m+1}\right)^2+\hf\left(\phi^m\right)^2 \right) \nabla \left(\hf\phi^{m+1} +\hf \phi^m - \phi^{m+\hf} \right)
	\nonumber
	\\
& + \phi^{m+\hf} \, \nabla\left(\hf\left(\phi^{m+1} \right)^2+\hf\left(\phi^m\right)^2 - \left(\phi^{m+\hf}\right)^2 \right)
	\nonumber
	\\
& +  \left(\hf\left(\phi^{m+1} \right)^2+\hf\left(\phi^m\right)^2 - \left(\phi^{m+\hf}\right)^2 \right)\nabla \phi^{m+\hf} . 
	\end{align}
By the triangle inequality, Young's inequality, and the embedding $H^1(\Omega) \hookrightarrow L^6(\Omega)$, we have
	\begin{align}
\norm{\nabla\sigma_4^{m+\hf}}{L^2} \le & \, \norm{\hf\phi^{m+1} +\hf \phi^m - \phi^{m+\hf}}{L^6} \norm{\nabla\left(\hf\left(\phi^{m+1}\right)^2+\hf\left(\phi^m\right)^2 \right)}{L^3} 
	\nonumber
	\\
& + \norm{\hf\left(\phi^{m+1}\right)^2+\hf\left(\phi^m\right)^2}{L^\infty} \norm{\nabla\left(\hf\phi^{m+1} +\hf \phi^m - \phi^{m+\hf}\right)}{L^2}
	\nonumber
	\\
& + \norm{\phi^{m+\hf}}{L^\infty} \norm{\nabla\left(\hf\left(\phi^{m+1} \right)^2+\hf\left(\phi^m\right)^2 - \left(\phi^{m+\hf}\right)^2 \right)}{L^2}
	\nonumber
	\\
& + \norm{\hf\left(\phi^{m+1} \right)^2+\hf\left(\phi^m\right)^2 - \left(\phi^{m+\hf}\right)^2}{L^6} \norm{\nabla\phi^{m+\hf}}{L^3} 
	\nonumber
	\\
\le & \, C\left\{\norm{\phi^{m+1}}{L^\infty}^2 + \norm{\phi^m}{L^\infty}^2 + \norm{\phi^{m+1}}{L^6}\norm{\nabla\phi^{m+1}}{L^6} + \norm{\phi^m}{L^6}\norm{\nabla\phi^m}{L^6} \right\}
	\nonumber
	\\
& \hspace{1in} \times \norm{\nabla \left(\hf\phi^{m+1} +\hf \phi^m - \phi^{m+\hf}\right)}{L^2}
	\nonumber
	\\
& + C\left\{\norm{\phi^{m+\hf}}{L^\infty} + \norm{\nabla\phi^{m+\hf}}{L^3} \right\} \times \norm{\hf\left(\phi^{m+1} \right)^2+\hf\left(\phi^m\right)^2 - \left(\phi^{m+\hf}\right)^2}{H^1}.
	\nonumber
	\\
& \mbox{}
	\end{align}
Using the assumed regularities \eqref{eq:higher-regularities} of the PDE solution, and appealing to the truncation error estimates \eqref{eq:truncation-4} and \eqref{eq:truncation-4-b}, the result follows.
	\end{proof}

	\begin{lem}
	\label{lem-sigma5-estimate}
Suppose that $(\phi, \mu)$ is a weak solution to \eqref{eq:weak-error-a} -- \eqref{eq:weak-error-b}, with the additional regularities \eqref{eq:higher-regularities}. Then, there exists a constant $C>0$ independent of $h$ and $\tau$, but possibly dependent upon $T$, such that, for any $h, \tau >0$,
	\begin{align}
\norm{\nabla \sigma_5^{m+\hf}}{L^2}^2 \le& \, C \norm{\nabla \e^{\phi,m+1}}{L^2}^2 + C \norm{\nabla \e^{\phi,m}}{L^2}^2,
	\end{align}
where $\e^{\phi,m} := \phi^{m} - \phih^{m}$.
	\end{lem}
	\begin{proof}
We begin with the detailed expansion
	\begin{align}
4\nabla\sigma_5^{m+\hf} =& \left\{\left(\phih^{m+1}\right)^2 + \left(\phih^m\right)^2+2 \phih^{m+1} \left(\phih^{m+1}+\phih^m \right)\right\} \nabla\left( \phih^{m+1}-\phi^{m+1}\right)
	\nonumber
	\\
 &+ \left\{\left(\phih^{m+1}\right)^2 + \left(\phih^m\right)^2+2 \phih^m \left(\phih^{m+1}+\phih^m \right) \right\} \nabla\left(\phih^m-\phi^m\right)
	\nonumber
	\\
&+ \bigg\{\nabla\left(\phi^{m+1}+\phi^m\right)\cdot \left(\phih^{m+1}+\phi^{m+1} \right) + 2 \nabla\phi^{m+1}\left(\phih^{m+1}+\phih^m \right)
	\nonumber
	\\
& \hspace{0.5in} +2  \phi^{m+1} \nabla\phi^{m+1} +2  \phi^m\nabla\phi^m\bigg\} \left(\phih^{m+1}-\phi^{m+1} \right)
	\nonumber
	\\
&+ \bigg\{ \nabla\left(\phi^{m+1}+\phi^m\right)\cdot \left(\phih^m+\phi^m \right) +2 \nabla\phi^m\left(\phih^{m+1}+\phih^m \right) 
	\nonumber
	\\
& \hspace{0.5in} +2  \phi^{m+1} \nabla\phi^{m+1} +2  \phi^m\nabla\phi^m\bigg\} \left(\phih^m-\phi^m \right) .
	\end{align}	
Then, using  the unconditional \emph{a priori} estimates in Lemmas~\ref{lem-a-priori-stability-trivial} and \ref{lem-a-priori-stability-ultimate}, the assumption that $\phi\in L^\infty\left(0,T;W^{1,6}(\Omega)\right)$, and the embedding $H^1(\Omega) \hookrightarrow L^6(\Omega)$ we have, for  any $0\le m\le M-1$,
	\begin{align}
\norm{\nabla \sigma_5^{m+\hf}}{L^2} \le & \, C\bigg\{ \norm{\phih^{m+1}}{L^\infty}^2 +\norm{\phih^m}{L^\infty}^2 \bigg\} \left( \norm{\nabla \e^{\phi,m+1}}{L^2}+\norm{\nabla \e^{\phi,m}}{L^2}\right) 
	\nonumber
	\\
&+  C\bigg\{\left( \norm{\nabla\phi^{m+1}}{L^6} +\norm{\nabla\phi^m}{L^6} \right)\cdot \left( \norm{\phi^{m+1}}{L^6} +\norm{\phi^m}{L^6}+ \norm{\phih^{m+1}}{L^6} +\norm{\phih^m}{L^6}\right) \bigg\}
	\nonumber
	\\
& \hspace{0.75in}\times \left( \norm{\e^{\phi,m+1}}{L^6}+\norm{\e^{\phi,m}}{L^6}\right) 
	\nonumber
	\\
\le & C  \norm{\nabla \e^{\phi,m+1}}{L^2}+C\norm{\nabla \e^{\phi,m}}{L^2} .
	\end{align}
	\end{proof}

	\begin{lem}
	\label{lem-sigma6-estimate}
Suppose that $(\phi, \mu)$ is a weak solution to \eqref{eq:weak-error-a} -- \eqref{eq:weak-error-b}, with the additional regularities \eqref{eq:higher-regularities}. Then, there exists a constant $C>0$ independent of $h$ and $\tau$ such that, for any $h, \tau >0$,
	\begin{align}
\norm{\nabla \sigma_6^{m+\hf}}{L^2}^2 \le& \,\gamma_m C \tau^3   \int_{t_{m-1}}^{t_m}\norm{\nabla\partial_{ss}\phi(s)}{L^2}^2 \, ds +  C \tau^3  \int_{t_m}^{t_{m+1}}\norm{\nabla\partial_{ss}\phi(s)}{L^2}^2 \, ds
	\nonumber
	\\
& + C\norm{\nabla \e^{\phi,m}}{L^2}^2+\gamma_mC\norm{\nabla \e^{\phi,m-1}}{L^2}^2 +\delta_{0,m} C h^{2q}\left| \phi_0 \right|_{H^{q+1}}^2,
	\end{align}
where $\e^{\phi,m} := \phi^{m} - \phih^{m}$ and $\delta_{k,\ell}$ is the Kronecker delta.
	\end{lem}
	\begin{proof}
For $m=0$, using the truncation error estimate~\eqref{eq:truncation-10} and a standard finite element estimate for the Ritz projection, we have
	\begin{align}
\norm{\nabla\sigma_6^{\hf}}{L^2}^2 & \le \,  2\norm{\nabla\left(\phi^{\hf} - \phi_0 -\frac{\tau}{2}\partial_t\phi(0)\right)}{L^2}^2 +2\norm{\nabla\left(\phi_0 - \phi_h^0\right)}{L^2}^2
	\nonumber
	\\
& \le \, 2\frac{\tau^3}{24} \int_{t_0}^{t_\hf} \norm{\nabla \partial_{ss} \phi(s)}{L^2}^2 ds + C h^{2q}\left| \phi_0 \right|_{H^{q+1}}^2,
	\end{align}
with the observation that $\phih^0 := R_h\phi_0$. For $1\le m\le M-1$, using the truncation error estimate~\eqref{eq:truncation-9}, we obtain
	\begin{equation}
\norm{\nabla\sigma_6^{m+\hf}}{L^2}^2 \le 3\frac{\tau^3}{6}  \int_{t_{m-1}}^{t_{m+1}}\norm{\nabla\partial_{ss}\phi(s)}{L^2}^2 \, ds + \frac{27}{4}\norm{\nabla\e^{\phi,m}}{L^2}^2 + \frac{3}{4}\norm{\nabla\e^{\phi,m-1}}{L^2}^2 .
	\end{equation}
	\end{proof}

We now proceed to estimate the terms on the right-hand-side of \eqref{eq:error-eq}. We will need the following technical lemmas.  The proof of the next result can be found in~\cite{diegel14}.
	\begin{lem}
	\label{lem-technical-L2-minus1}
Suppose $g \in H^1(\Omega)$, and $v \in \Soh$.  Then
	\begin{equation}
\left|\iprd{g}{v}\right| \le C \norm{\nabla g}{L^2} \, \norm{v}{-1,h}  ,
	\end{equation}
for some $C>0$ that is independent of $h$.
	\end{lem}
	\begin{lem}
	\label{lem-error-rhs-control}
Suppose that $(\phi, \mu)$ is a weak solution to \eqref{eq:weak-error-a} -- \eqref{eq:weak-error-b}, with the additional regularities \eqref{eq:higher-regularities}.  Then, for any $h$, $\tau >0$ and any $\alpha > 0$ there exists a constant $C=C(\alpha,T)>0$, independent of $h$ and $\tau$, such that, for $0 \le m \le M-1$,
	\begin{align}
	\nonumber
&\frac{\varepsilon}{2} \, \aiprd{\eh^{\phi,m+1} + \eh^{\phi,m}}{\dtau \eh^{\phi, m+\frac12}} + \frac{\gamma_m\varepsilon\tau^2}{4} \, \aiprd{\ddtau\eh^{\phi,m}}{\dtau \eh^{\phi, m+\frac12}} + \frac{\varepsilon}{2} \, \norm{\nabla \eh^{\mu, m+\frac12}}{L^2}^2 
	\\
&\quad\le \, C \norm{\nabla \eh^{\phi,m+1}}{L^2}^2 + C \norm{\nabla \eh^{\phi,m}}{L^2}^2 + \gamma_m C \norm{\nabla \eh^{\phi,m-1}}{L^2}^2 + \alpha \norm{\dtau \eh^{\phi,m+\frac12}}{-1,h}^2 + C \mathcal{R}^{m+\hf},
	\label{eq:right-hand-side-estimate}
	\end{align}
where	
	\begin{align}
\mathcal{R}^{m+\hf} = &\,  \frac{h^{2q+2}}{\tau} \int_{t_m}^{t_{m+1}} \norm{\partial_s\phi(s)}{H^{q+1}}^2  ds + h^{2q}\left| \mu^{m+\frac12} \right|_{H^{q+1}}^2 
	\nonumber
	\\
& +  h^{2q}\left| \phi^{m+1} \right|_{H^{q+1}}^2  +  h^{2q}\left| \phi^{m} \right|_{H^{q+1}}^2 + \gamma_m h^{2q}\left| \phi^{m-1} \right|_{H^{q+1}}^2 
	\nonumber
	\\
& +  \tau^3 \int_{t_m}^{t_{m+1}}\norm{\partial_{sss}\phi(s)}{L^2}^2 ds  +  \tau^3 \int_{t_m}^{t_{m+1}} \norm{\partial_{ss} \phi^2(s)}{H^1}^2 ds 
	\nonumber
	\\
& + \gamma_m \tau^3 \int_{t_{m-1}}^{t_m}\norm{\nabla\partial_{ss}\phi(s)}{L^2}^2 \, ds +  \tau^3 \int_{t_m}^{t_{m+1}}\norm{\nabla\partial_{ss}\phi(s)}{L^2}^2 \, ds
	\nonumber
	\\
& +  \gamma_m \tau^3 \int_{t_{m-1}}^{t_m} \norm{\nabla \Delta \partial_{ss} \phi(s)}{L^2}^2 ds +  \tau^3 \int_{t_m}^{t_{m+1}} \norm{\nabla \Delta \partial_{ss} \phi(s)}{L^2}^2 ds .
	\label{eq:consistency}
	\end{align}

	\end{lem}

	\begin{proof}
Define, for $0\le m\le M-1$, time-dependent spatial mass average
	\begin{equation}
\overline{\e_h^{\mu,m+\hf}} := |\Omega|^{-1}\iprd{\e_h^{\mu,m+\hf}}{1}.
	\end{equation}
Using the Cauchy-Schwarz inequality, the Poincar\'{e} inequality, with the fact that 
	\[
\iprd{\sigma^{m+\frac12}_1 + \sigma^{m+\frac12}_2}{1} = 0,
	\]
and the local truncation error estimates~\eqref{eq:truncation-1} and \eqref{eq:truncation-2}, we get the following estimate:
	\begin{align}
\left|\iprd{\sigma^{m+\frac12}_1 + \sigma^{m+\frac12}_2}{\eh^{\mu,m+\frac12}}\right| =& \left|\iprd{\sigma^{m+\frac12}_1 + \sigma^{m+\frac12}_2}{\eh^{\mu,m+\frac12}-\overline{\e_h^{\mu,m+\hf}}}\right|
	\nonumber
	\\
\le & \, \norm{\sigma^{m+\frac12}_1 + \sigma^{m+\frac12}_2}{L^2} \norm{\eh^{\mu,m+\frac12}-\overline{\e_h^{\mu,m+\hf}}}{L^2}
	\nonumber
	\\
\le& \, C \norm{\sigma^{m+\frac12}_1 + \sigma^{m+\frac12}_2}{L^2} \norm{\nabla \eh^{\mu,m+\frac12}}{L^2}
	\nonumber
	\\
\le& \, C\norm{\sigma^{m+\frac12}_1}{L^2}^2 + C\norm{\sigma^{m+\frac12}_2}{L^2}^2 + \frac{\varepsilon}{2} \norm{\nabla \eh^{\mu,m+\frac12}}{L^2}^2
	\nonumber
	\\
\le& \, C  \frac{h^{2q+2}}{\tau} \int_{t_m}^{t_{m+1}} \norm{\partial_s\phi(s)}{H^{q+1}}^2  ds
	\nonumber
	\\
& + C\frac{\tau^3}{640} \int_{t_m}^{t_{m+1}}\norm{\partial_{sss}\phi(s)}{L^2}^2 ds   + \frac{\varepsilon}{2} \norm{\nabla \eh^{\mu,m+\frac12}}{L^2}^2.
	\label{eq:error-estimate-1}
	\end{align}
Standard finite element approximation theory shows that
	\begin{equation*}
\norm{\nabla \eA^{\mu,m+\frac12}}{L^2} = \norm{\nabla \left(R_h \mu^{m+\frac12} - \mu^{m+\frac12}\right)}{L^2} \leq C h^q\left| \mu^{m+\frac12} \right|_{H^{q+1}} .
	\end{equation*}
Applying Lemma~\ref{lem-technical-L2-minus1} and the last estimate, we have
	\begin{align}
\left|\iprd{\eA^{\mu,m+\frac12}}{\dtau \eh^{\phi,m+\frac12}}\right| \le& \, C \norm{\nabla\eA^{\mu,m+\frac12}}{L^2} \, \norm{\dtau \eh^{\phi,m+\frac12}}{-1,h} 
	\nonumber
	\\
\le& \, C h^q\left| \mu^{m+\frac12} \right|_{H^{q+1}} \norm{\dtau \eh^{\phi,m+\frac12}}{-1,h}
	\nonumber
	\\
\le& \, Ch^{2q}\left| \mu^{m+\frac12} \right|_{H^{q+1}}^2 + \frac{\alpha}{6} \norm{\dtau \eh^{\phi,m+\frac12}}{-1,h}^2 .
	\label{eq:error-estimate-2}
	\end{align}
Using Lemma~\ref{lem-technical-L2-minus1} and estimate \eqref{eq:truncation-3}, we find
	\begin{align}
\varepsilon \, \aiprd{\sigma_3^{m+\frac12}}{\dtau \eh^{\phi,m+\frac12}} &= -\varepsilon \, \iprd{\Delta \sigma_3^{m+\hf}}{\dtau \eh^{\phi,m+\frac12}}
	\nonumber
	\\
&\le C \norm{\nabla \Delta \sigma_3^{m+\hf}}{L^2} \norm{\dtau \eh^{\phi,m+\frac12}}{-1,h}
	\nonumber
	\\
&\le C \, \frac{\tau^3}{96} \int_{t_m}^{t_{m+1}} \norm{\nabla \Delta \partial_{ss} \phi(s)}{L^2}^2 ds + \frac{\alpha}{6} \norm{\dtau \eh^{\phi,m+\frac12}}{-1,h}^2.
	\nonumber
	\\
& \mbox{}
	\label{eq:error-estimate-3}
	\end{align}
	
Now, using Lemmas~\ref{lem-sigma4-estimate} and \ref{lem-technical-L2-minus1}, we obtain
	\begin{align}
\varepsilon^{-1}\left|\iprd{\sigma_4^{m+\hf}}{\dtau \eh^{\phi,m+\frac12}}\right| \le& \, C \norm{\nabla \sigma_4^{m+\hf}}{L^2} \, \norm{\dtau \eh^{\phi,m+\frac12}}{-1,h}
	\nonumber
	\\
\le& \, C \norm{\nabla \sigma_4^{m+\hf}}{L^2}^2 +  \frac{\alpha}{6}  \norm{\dtau \eh^{\phi,m+\frac12}}{-1,h}^2
	\nonumber
	\\
\le& \, C\tau^3 \int_{t_m}^{t_{m+1}} \norm{\nabla \partial_{ss} \phi(s)}{L^2}^2 ds 
	\nonumber
	\\
& + C \tau^3 \int_{t_m}^{t_{m+1}} \norm{\partial_{ss} \phi^2(s)}{H^1}^2 ds + \frac{\alpha}{6}  \norm{\dtau \eh^{\phi,m+\frac12}}{-1,h}^2.
	\label{eq:error-estimate-4}
	\end{align}
Similarly, using Lemmas~\ref{lem-sigma5-estimate} and \ref{lem-technical-L2-minus1}, the relation $\e^{\phi,m+1} =\eA^{\phi,m+1}+\eh^{\phi,m+1}$, and a standard finite element error estimate, we arrive at
	\begin{align}
\varepsilon^{-1}\left|\iprd{\sigma_5^{m+\hf}}{\dtau \eh^{\phi,m+\frac12}}\right| \le& \, C \norm{\nabla \sigma_5^{m+\hf}}{L^2}^2 +  \frac{\alpha}{6}  \norm{\dtau \eh^{\phi,m+\frac12}}{-1,h}^2
	\nonumber
	\\
\le& \, C \norm{\nabla \e^{\phi,m+1}}{L^2}^2 + C \norm{\nabla \e^{\phi,m}}{L^2}^2 + \frac{\alpha}{6}  \norm{\dtau \eh^{\phi,m+\frac12}}{-1,h}^2
	\nonumber
	\\
\le& \, C \norm{\nabla \eA^{\phi,m+1}}{L^2}^2 +C \norm{\nabla \eh^{\phi,m+1}}{L^2}^2 + C \norm{\nabla \eA^{\phi,m}}{L^2}^2 
	\nonumber
	\\
&+C \norm{\nabla \eh^{\phi,m}}{L^2}^2 + \frac{\alpha}{6}  \norm{\dtau \eh^{\phi,m+\frac12}}{-1,h}^2 
	\nonumber
	\\
\le& \, C h^{2q}\left| \phi^{m+1} \right|_{H^{q+1}}^2 + C \norm{\nabla \eh^{\phi,m+1}}{L^2}^2 + C h^{2q}\left| \phi^{m} \right|_{H^{q+1}}^2 
	\nonumber
	\\
&+ C \norm{\nabla \eh^{\phi,m}}{L^2}^2 + \frac{\alpha}{6}  \norm{\dtau \eh^{\phi,m+\frac12}}{-1,h}^2.
	\label{eq:error-estimate-5}
	\end{align}
Applying Lemmas~\ref{lem-sigma6-estimate} and \ref{lem-technical-L2-minus1}, the relation $\e^{\phi,m+1} =\eA^{\phi,m+1}+\eh^{\phi,m+1}$, and a standard finite element error estimate,
	\begin{align}
\varepsilon^{-1}\left|\iprd{\sigma_6^{m+\hf}}{\dtau \eh^{\phi,m+\frac12}}\right| \le& \, C \norm{\nabla \sigma_6^{m+\hf}}{L^2}^2 +  \frac{\alpha}{6}  \norm{\dtau \eh^{\phi,m+\frac12}}{-1,h}^2
	\nonumber
	\\
\le & \, C \tau^3  \left( \gamma_m\int_{t_{m-1}}^{t_m}\norm{\nabla\partial_{ss}\phi(s)}{L^2}^2 \, ds + \int_{t_m}^{t_{m+1}}\norm{\nabla\partial_{ss}\phi(s)}{L^2}^2 \, ds \right) 
	\nonumber
	\\
& + C\norm{\nabla \e_h^{\phi,m}}{L^2}^2+C\gamma_m\norm{\nabla \e_h^{\phi,m-1}}{L^2}^2 
	\nonumber
	\\
& + C h^{2q}\left| \phi^m \right|_{H^{q+1}}^2 + C\gamma_m h^{2q}\left| \phi^{m-1} \right|_{H^{q+1}}^2 + \frac{\alpha}{6}  \norm{\dtau \eh^{\phi,m+\frac12}}{-1,h}^2.
	\label{eq:error-estimate-6}
	\end{align}
	
To finish up, using \eqref{eq:truncation-4},
	\begin{align}
\frac{\gamma_m\varepsilon\tau^2}{4} \, \aiprd{\ddtau\phi^m}{\dtau \eh^{\phi, m+\frac12}}  \le &\, C \gamma_m\frac{\tau^3}{3} \int_{t_{m-1}}^{t_m} \norm{\nabla \Delta \partial_{ss} \phi(s)}{L^2}^2 ds + \frac{\alpha}{6} \norm{\dtau \eh^{\phi,m+\frac12}}{-1,h}^2 .
	\label{eq:error-estimate-7}
	\end{align}

Combining the estimates \eqref{eq:error-estimate-1} -- \eqref{eq:error-estimate-7} with the error equation \eqref{eq:error-eq}, the result follows.
	\end{proof}

	\begin{lem}
	\label{lem-1,h-error-estimate}
Suppose that $(\phi, \mu)$ is a weak solution to \eqref{eq:weak-error-a} -- \eqref{eq:weak-error-b}, with the additional regularities \eqref{eq:higher-regularities}.  Then, for any $h, \tau >0$, there exists a constant $C>0$, independent of $h$ and $\tau$, such that
	\begin{equation}
\norm{\dtau \eh^{\phi,m+\frac12}}{-1,h}^2 \le 2 \, \varepsilon^2 \norm{\nabla \eh^{\mu,m+\frac12}}{L^2}^2 + C\mathcal{R}^{m+\hf},
 	\label{eq:-1,h-error-estimate}
	\end{equation}
where $\mathcal{R}^{m+\hf}$ is the consistency term given in \eqref{eq:consistency}.
	\end{lem}
	
	\begin{proof}
Setting $\nu = \mathsf{T}_h\left(\dtau \eh^{\phi,\frac12} \right)$ in \eqref{eq:error-a-initial} and $\nu = \mathsf{T}_h\left(\dtau \eh^{\phi,m+\frac12} \right)$ in \eqref{eq:error-a} and combining, we have
	\begin{align}
\norm{\dtau \eh^{\phi,m+\frac12}}{-1,h}^2  =& \, - \varepsilon\, \aiprd{\eh^{\mu,m+\frac12}}{\mathsf{T}_h\left(\dtau \eh^{\phi,m+\frac12} \right)} + \iprd{\sigma^{m+\frac12}_1 + \sigma^{m+\frac12}_2}{\mathsf{T}_h\left(\dtau \eh^{\phi,m+\frac12} \right)} 
	\nonumber
	\\
=& \, -\varepsilon \iprd{\eh^{\mu,m+\frac12}}{\dtau \eh^{\phi,m+\frac12}} + \iprd{\sigma^{m+\frac12}_1 + \sigma^{m+\frac12}_2}{\mathsf{T}_h\left(\dtau \eh^{\phi,m+\frac12} \right)} 
	\nonumber
	\\
\le& \, \varepsilon \norm{\nabla \eh^{\mu,m+\frac12}}{L^2} \norm{\dtau \eh^{\phi,m+\frac12}}{-1,h} + \norm{\sigma^{m+\frac12}_1 + \sigma^{m+\frac12}_2}{L^2} \norm{\mathsf{T}_h\left(\dtau \eh^{\phi,m+\frac12} \right)}{L^2}
	\nonumber
	\\
\le& \, \varepsilon^2 \norm{\nabla \eh^{\mu,m+\frac12}}{L^2}^2 + \frac{1}{4} \norm{\dtau \eh^{\phi,m+\frac12}}{-1,h}^2 
	\nonumber
	\\
& + C \norm{\sigma^{m+\frac12}_2 + \sigma^{m+\frac12}_1}{L^2}^2 + \frac{1}{4} \norm{\dtau \eh^{\phi, m+\frac12}}{-1,h}^2
	\nonumber
	\\
\le& \, \varepsilon^2 \norm{\nabla \eh^{\mu,m+\frac12}}{L^2}^2 + \frac12 \norm{\dtau \eh^{\phi,m+\frac12}}{-1,h}^2 + C\mathcal{R}^{m+\hf} ,
	\end{align}
for $0 \le m \le M-1$ and where we have used Lemma~\ref{lem-truncation-errors}. The result now follows.
	\end{proof}

	\begin{lem}
	\label{lem-error-rhs-eh-control}
Suppose that $(\phi, \mu)$ is a weak solution to \eqref{eq:weak-error-a} -- \eqref{eq:weak-error-b}, with the additional regularities \eqref{eq:higher-regularities}.  Then, for any $h$, $\tau >0$, there exists a constant $C>0$, independent of $h$ and $\tau$, but possibly dependent upon $T$, such that 
	\begin{align}
&\frac{\varepsilon}{2} \, \aiprd{\eh^{\phi,m+1} + \eh^{\phi,m}}{\dtau \eh^{\phi, m+\frac12}} + \frac{\gamma_m\tau^2\varepsilon}{4} \, \aiprd{\ddtau \eh^{\phi,m} }{\dtau \eh^{\phi, m+\frac12}} + \frac{\varepsilon}{4} \, \norm{\nabla \eh^{\mu, m+\frac12}}{L^2}^2 
	\nonumber
	\\
&\quad\le \, C \norm{\nabla \eh^{\phi, m+1}}{L^2}^2 + C \norm{\nabla \eh^{\phi, m}}{L^2}^2 + \gamma_mC \norm{\nabla \eh^{\phi, m-1}}{L^2}^2 + C\mathcal{R}^{m+1} .
	\end{align}
	\end{lem}
	\begin{proof}
This follows upon combining the last two lemmas and choosing $\alpha$ in \eqref{eq:right-hand-side-estimate} appropriately.
	\end{proof}

Using the last lemma, we are ready to show the main convergence result for our second-order convex-splitting scheme.
	\begin{thm}
	\label{thm-error-estimate}
Suppose $(\phi, \mu)$ is a weak solution to \eqref{eq:weak-error-a} -- \eqref{eq:weak-error-b}, with the additional regularities \eqref{eq:higher-regularities}.  Then, provided $0<\tau <\tau_0$, for some $\tau_0$ sufficiently small,
	\begin{align}
\max_{0 \le m \le M-1} \norm{\nabla\eh^{\phi,m+1}}{L^2}^2 + \tau\sum_{m=0}^{M-1} \norm{\nabla\eh^{\mu,m+\frac12}}{L^2}^2 &\le C(T)(\tau^4+h^{2q}) 
	\end{align}
for some $C(T)>0$ that is independent of $\tau$ and $h$.
	\end{thm}

	\begin{proof}
Using Lemma~\ref{lem-error-rhs-eh-control}, we have
	\begin{align}
\frac{1}{2 \tau} \left( \norm{\nabla \eh^{\phi,m+1}}{L^2}^2 - \norm{\nabla \eh^{\phi,m}}{L^2}^2 \right) +  \frac14 \norm{\nabla\eh^{\mu,m+\frac12}}{L^2}^2 \hspace{0.6in}&
	\nonumber
	\\
  + \, \frac{\gamma_m}{8 \tau} \left( \norm{\nabla \eh^{\phi,m+1} - \nabla \eh^{\phi,m}}{L^2}^2 - \norm{\nabla \eh^{\phi,m} - \nabla \eh^{\phi,m-1}}{L^2}^2 \right) \le & \, C \norm{\nabla \eh^{\phi, m+1}}{L^2}^2 + C \norm{\nabla \eh^{\phi, m}}{L^2}^2  
	\nonumber
	\\
& + \gamma_m C \norm{\nabla \eh^{\phi, m-1}}{L^2}^2 + C\mathcal{R}^{m+\hf}.
	\nonumber
	\\
& \mbox{}
	\label{eq:error-sum}
	\end{align}
Letting $m=0$ in the previous equation and noting that $\eh^{\phi,0} \equiv 0$ and $\gamma_0 = 0$, then
	\begin{align}
\frac{1}{2 \tau} \norm{\nabla \eh^{\phi,1}}{L^2}^2 + \frac14 \norm{\nabla\eh^{\mu,\frac12}}{L^2}^2 \le  \, C_1 \norm{\nabla \eh^{\phi, 1}}{L^2}^2 + C\mathcal{R}^{\hf} .
	\end{align}
If $0< \tau \le \tau_0:= \frac{1}{2C_1} < \frac{1}{C_1}$, it follows from the last estimate that
	\begin{align}
	\norm{\nabla \eh^{\phi,1}}{L^2}^2 + \frac{\tau}{2} \norm{\nabla\eh^{\mu,\frac12}}{L^2}^2 \le \tau \, C \mathcal{R}^{\hf} \le C(\tau^4+h^{2q}),
	\label{eq:error-estimate-initial}
	\end{align}
where we have used the regularity assumptions to conclude $\tau \, C \mathcal{R}^{\hf}\le C(\tau^4+h^{2q})$. Now, applying $\tau\sum_{m=0}^\ell$ to \eqref{eq:error-sum}, 
	\begin{align}
\norm{\nabla\eh^{\phi,\ell+1}}{L^2}^2 + \frac{\tau}{2} \sum_{m=0}^{\ell} \norm{\nabla\eh^{\mu,m+\frac12}}{L^2}^2 &\le  C\tau\sum_{m=0}^\ell \mathcal{R}^{m+\hf} + C_2\tau\sum_{m=0}^\ell\norm{\nabla\eh^{\phi,m+1}}{L^2}^2 + \frac14 \norm{\nabla \eh^{\phi,1}}{L^2}^2.
	\label{eq:estimate-before-gronwall}
	\end{align}
If $0< \tau \le \tau_0:= \frac{1}{2C_2} < \frac{1}{C_2}$, it follows from the last estimate that
	\begin{align}
\norm{\nabla\eh^{\phi,\ell+1}}{L^2}^2  \le& \, C\tau\sum_{m=0}^\ell \mathcal{R}^{m+\hf} + \frac{C_2\tau}{1-C_2\tau} \sum_{m=0}^{\ell}\norm{\nabla\eh^{\phi,m}}{L^2}^2 + \frac14 \norm{\nabla \eh^{\phi,1}}{L^2}^2
	\nonumber
	\\
\le& \, C(\tau^4+h^{2q}) + C\tau \sum_{m=0}^{\ell}\norm{\nabla\eh^{\phi,m}}{L^2}^2 ,
	\label{eq:pre-gronwall}
	\end{align}
where we have used \eqref{eq:error-estimate-initial} and the regularity assumptions to conclude $\tau\sum_{m=0}^{M-1} \mathcal{R}^{m+\hf}\le C(\tau^4+h^{2q})$. Appealing to the discrete Gronwall inequality, it follows that, for any $0 < \ell \le M-1$,
	\begin{equation}
\norm{\nabla\eh^{\phi,\ell+1}}{L^2}^2 \le C(T)(\tau^4+h^{2q}).
	\label{eq:post-gronwall}
	\end{equation}

Considering estimates \eqref{eq:error-estimate-initial}, \eqref{eq:estimate-before-gronwall}, and \eqref{eq:post-gronwall} we get the desired result. 
	\end{proof}
	
	\begin{rem}
From here it is straightforward to establish an optimal error estimate of the form
	\begin{align}
\max_{0\le m \le M-1} \norm{\nabla\e^{\phi,\ell+1}}{L^2}^2 + \tau\sum_{m=0}^{M-1} \norm{\nabla\e^{\mu,m+\frac12}}{L^2}^2  \le C(T)(\tau^4+h^{2q}) 
	\end{align}
using $\ephi = \eAphi + \ehphi$, \emph{et cetera}, the triangle inequality, and the standard spatial approximations. We omit the details for the sake of brevity.
	\end{rem}

	\section{Numerical Experiments}
	\label{sec:numerical-experiments}
	
In this section, we provide some numerical experiments to gauge the accuracy and reliability of the fully discrete finite element method developed in the previous sections. We use a square domain $\Omega = (0,1)^2\subset \mathbb{R}^2$ and take ${\mathcal T}_h$ to be a regular triangulation of $\Omega$ consisting of right isosceles triangles. To refine the mesh, we assume that ${\mathcal T}_{\ell}, \ {\ell} = 0, 1, ..., L$, is an hierarchy of nested triangulations of $\Omega$ where ${\mathcal T}_{\ell}$, is obtained by subdividing the triangles of ${\mathcal T}_{\ell -1}$ into four congruent sub-triangles. Note that $h_{\ell -1} = 2h_{\ell}, \ {\ell} = 1, ..., L$, and that $\{{\mathcal T}_{\ell}\}$ is a quasi-uniform family. (We use a family of meshes ${\mathcal T}_h$ such that no triangle in the mesh has more than one edge on the boundary.)  We use the ${\mathcal P}_2$ finite element space for the phase field and chemical potential.  In short, we take $q=2$.

We solve the scheme \eqref{eq:scheme-a} -- \eqref{eq:scheme-b-initial} with $\epsilon = 6.25 \times 10^{-2}$. The initial data for the phase field is taken to be
	\begin{equation}
\phi_{h}^0 = \mathcal{I}_h\left\{ \frac{1}{2}\Big(1.0-\cos(4.0\pi x)\Big)\cdot \Big(1.0-\cos(2.0\pi y)\Big)-1.0\right\} ,
	\end{equation}
where $\mathcal{I}_h :   H^2\left(\Omega\right) \to S_h$ is the standard nodal interpolation operator. Recall that our analysis does not specifically cover the use of the operator $\mathcal{I}_h$ in the initialization step.  But, since the error introduced by its use is optimal, a slight modification of the analysis shows that this will lead to optimal rates of convergence overall.  (See Remark~\ref{rem:initial-projection}.) To solve the system of equations above numerically, we are using the finite element libraries from the FEniCS Project~\cite{fenics12}.  

	\begin{table}[h!]
	\centering
	\begin{tabular}{ccccccccc}
$h_c$ & $h_f$ & $\norm{\delta_\phi}{H^1}$ & rate & $\norm{\delta_\mu}{H^1}$ & rate &
	\\
	\hline
$\nicefrac{\sqrt{2}}{16}$ & $\nicefrac{\sqrt{2}}{32}$ & $1.148\times 10^{-1}$ & -- & $1.307\times 10^{-1}$ & -- & 
	\\
$\nicefrac{\sqrt{2}}{32}$ & $\nicefrac{\sqrt{2}}{64}$ & $2.939\times 10^{-2}$ & 1.95 & $3.299\times 10^{-2}$ & 1.98 & 
	\\
$\nicefrac{\sqrt{2}}{64}$ & $\nicefrac{\sqrt{2}}{128}$ & $7.468\times 10^{-3}$ & 1.97 & $8.295\times 10^{-3}$ & 1.99 & 
	\\
$\nicefrac{\sqrt{2}}{128}$ & $\nicefrac{\sqrt{2}}{256}$ & $1.913\times 10^{-3}$ & 1.95 & $2.087\times 10^{-3}$ & 1.99 & 
	\\
	\hline
	\end{tabular}
\caption{$H^1$ Cauchy convergence test. The final time is $T = 4.0\times 10^{-1}$, and the refinement path is taken to be $\tau = .001\sqrt{2}h$ with $\varepsilon =6.25\times 10^{-2}$.  The Cauchy difference is defined via $\delta_\phi := \phi_{h_f}-\phi_{h_c}$, where the approximations are evaluated at time $t=T$, and analogously for $\delta_\mu$. (See the discussion in the text.) Since $q=2$, \emph{i.e.}, we use ${\mathcal P}_2$ elements for these variables, the norm of the Cauchy difference at $T$ is expected to be $\mathcal{O}(\tau_f^2)+\mathcal{O}\left(h_f^2\right) = \mathcal{O}\left(h_f^2\right)$.}
	\label{tab1}
	\end{table}
	
Note that source terms are not naturally present in the system of equations \eqref{eq:CH-mixed-a-alt} -- \eqref{eq:CH-mixed-c-alt}. Therefore, it is somewhat artificial to add them to the equations in attempt to manufacture exact solutions.  To get around the fact that we do not have possession of exact solutions, we measure error by a different means.  Specifically, we compute the rate at which the Cauchy difference, $\delta_\zeta := \zeta^{M_f}_{h_f} - \zeta^{M_c}_{h_c}$, converges to zero, where $h_f=2h_c$, $\tau_f = 2\tau_c$, and $\tau_fM_f = \tau_cM_c=T$. Then, using a linear refinement path, \emph{i.e.}, $\tau = Ch$, and assuming $q = 2$, we have
	\begin{equation}
\norm{\delta_\zeta}{H^1} = \norm{\zeta^{M_f}_{h_f} - \zeta^{M_c}_{h_c}}{H^1} \le \norm{\zeta^{M_f}_{h_f}-\zeta(T)}{H^1}+ \norm{\zeta^{M_c}_{h_c}-\zeta(T)}{H^1} = \mathcal{O}(h_f^q+\tau_f^2) = \mathcal{O}(h_f^2).
	\end{equation}
The results of the $H^1$ Cauchy error analysis are found in Table~\ref{tab1} and confirm second-order convergence in this case. Additionally, we have proved that (at the theoretical level) the modified energy is non-increasing at each time step.  This is observed in our computations, but, for the sake of brevity, we will suppress an extensive discussion of numerical energy dissipation.

	\section*{Acknowledgements} This work is supported in part by the grants NSF DMS-1115420, 1418689, and NSFC 11271281 (C.~Wang), and NSF-DMS 1115390 and 1418692 (S.~Wise).
	
	\bibliographystyle{plain}
	\bibliography{SOCH}

\begin{thebibliography}{10}

\bibitem{aristotelous13}
A.~Aristotelous, O.~Karakashian, and S.M. Wise.
\newblock A mixed discontinuous {Galerkin}, convex splitting scheme for a
  modified {Cahn-Hilliard} equation and an efficient nonlinear multigrid
  solver.
\newblock {\em Discrete Cont. Dyn Sys. B}, 18(9), 2013.

\bibitem{aristotelous14}
A.~Aristotelous, O.~Karakashian, and S.M. Wise.
\newblock Adaptive, second-order in time, primitive-variable discontinuous
  {Galerkin} schemes for a {Cahn-Hilliard} equation with a mass source term.
\newblock {\em IMA J. Numer. Anal.}, (to appear), 2014.

\bibitem{cahn61}
J.W. Cahn.
\newblock On spinodal decomposition.
\newblock {\em Acta Metall.}, 9:795, 1961.

\bibitem{cahn58}
J.W. Cahn and J.E. Hilliard.
\newblock Free energy of a nonuniform system. {I}. {I}nterfacial free energy.
\newblock {\em J. Chem. Phys.}, 28:258, 1958.

\bibitem{chen98}
L.Q. Chen and J.~Shen.
\newblock Applications of semi-implicit fourier-spectral method to phase field
  equations.
\newblock {\em Computer Physics Communications}, 108(2):147--158, 1998.

\bibitem{diegel14}
A.~Diegel, X.~Feng, and S.M. Wise.
\newblock Analysis of a mixed finite element method for a
  {Cahn-Hilliard-Darcy-Stokes} system.
\newblock {\em SIAM J. Numer. Anal.}, (to appear), 2014.

\bibitem{du91}
Q.~Du and R.~Nicolaides.
\newblock Numerical analysis of a continuum model of a phase transition.
\newblock {\em SIAM J. Num. Anal.}, 28:1310--1322, 1991.

\bibitem{elliott86}
C.~M. Elliott and S.~Zheng.
\newblock On the {C}ahn-{H}illiard equation.
\newblock {\em Arch. Rational Mech. Anal.}, 96:339--357, 1986.

\bibitem{elliott89a}
C.M. Elliott.
\newblock The {Cahn-Hilliard} model for the kinetics of phase separation.
\newblock In J.F. Rodrigues, editor, {\em Mathematical Models for Phase Change
  Problems: Proceedings of the {European} Workshop held at {\'O}bidos,
  {Portugal}, {October} 1-3, 1988}, International Series of Numerical
  Mathematics, pages 35--73, Berlin, 1989. Birkh{\"a}user Verlag.

\bibitem{elliott89b}
C.M. Elliott, D.A French, and F.A. Milner.
\newblock A second-order splitting method for the {Cahn-Hilliard} equation.
\newblock {\em Numer. Math.}, 54:575--590, 1989.

\bibitem{elliott92}
C.M. Elliott and S.~Larsson.
\newblock Error estimates with smooth and nonsmooth data for a finite element
  method for the {Cahn-Hilliard} equation.
\newblock {\em Math. Comp.}, 58:603--630, 1992.

\bibitem{elliott93}
C.M. Elliott and A.M. Stuart.
\newblock The global dynamics of discrete semilinear parabolic equations.
\newblock {\em SIAM J. Numer. Anal}, 30:1622--1663, 1993.

\bibitem{elliott96}
C.M. Elliott and A.M. Stuart.
\newblock Viscous {C}ahn-{H}illiard equation. {II}. {A}nalysis.
\newblock {\em J. Diff. Eq.}, 128:387--414, 1996.

\bibitem{eyre98}
D.~Eyre.
\newblock Unconditionally gradient stable time marching the {C}ahn-{H}illiard
  equation.
\newblock In J.~W. Bullard, R.~Kalia, M.~Stoneham, and L.Q. Chen, editors, {\em
  Computational and Mathematical Models of Microstructural Evolution},
  volume~53, pages 1686--1712, Warrendale, PA, USA, 1998. Materials Research
  Society.

\bibitem{feng06}
X.~Feng.
\newblock Fully discrete finite element approximations of the
  {Navier-Stokes-Cahn-Hilliard} diffuse interface model for two-phase fluid
  flows.
\newblock {\em SIAM J. Numer. Anal.}, 44:1049--1072, 2006.

\bibitem{feng04}
X.~Feng and A.~Prohl.
\newblock Error analysis of a mixed finite element method for the
  {Cahn-Hilliard} equation.
\newblock {\em Numer. Math.}, 99:47--84, 2004.

\bibitem{feng12}
X.~Feng and S.M. Wise.
\newblock Analysis of a {Darcy-Cahn-Hilliard} diffuse interface model for the
  {Hele-Shaw} flow and its fully discrete finite element approximation.
\newblock {\em SIAM J. Numer. Anal.}, 50:1320--1343, 2012.

\bibitem{fratzl99}
P.~Fratzl, O.~Penrose, and J.L. Lebowitz.
\newblock Modelling of phase separation in alloys with coherent elastic misfit.
\newblock {\em J. Stat. Physics}, 95:1429--1503, 1999.

\bibitem{furihata01}
D.~Furihata.
\newblock A stable and conservative finite difference scheme for the
  {Cahn-Hilliard} equation.
\newblock {\em Numer. Math.}, 87:675--699, 2001.

\bibitem{garcke05}
H.~Garcke and U.~Weikard.
\newblock Numerical approximation of the {Cahn-Larch\'{e}} equation.
\newblock {\em Numer. Math.}, 100:639--662, 2005.

\bibitem{grun13}
G.~Gr\"{u}n.
\newblock On convergent schemes for diffuse interface models for two-phase flow
  of incompressible fluids with general mass densities.
\newblock {\em SIAM J. Numer. Anal.}, 51(6):3036--3061, 2013.

\bibitem{grun14}
G.~Gr\"{u}n and F.~Klingbeil.
\newblock Two-phase flow with mass density contrast: Stable schemes for a
  thermodynamic consistent and frame indifferent diffuse-interface model.
\newblock {\em J. Comput. Phys.}, 257:708--725, 2014.

\bibitem{guan13}
Z.~Guan, C.~Wang, and S.~M. Wise.
\newblock A convergent convex splitting scheme for the periodic nonlocal
  {Cahn-Hilliard} equation.
\newblock {\em Numerische Mathematik}, pages 1--30, 2013.

\bibitem{guo14}
J.~Guo, C.~Wang, S.~Wise, and X.~Yue.
\newblock An $h^2$ convergence of a second-order convex-splitting, finite
  difference scheme for the three-dimensional {Cahn-Hilliard} equation.
\newblock {\em Commun Math Sci}, 2014.
\newblock In review.

\bibitem{he07}
Y.~He, Y.~Liu, and T.~Tang.
\newblock On large time-stepping methods for the {Cahn-Hilliard} equation.
\newblock {\em Appl. Numer. Math.}, 57(4):616--628, 2006.

\bibitem{hu09}
Z.~Hu, S.M. Wise, C.~Wang, and J.S. Lowengrub.
\newblock Stable and efficient finite-difference nonlinear-multigrid schemes
  for the phase-field crystal equation.
\newblock {\em J. Comput. Phys.}, 228:5323--5339, 2009.

\bibitem{kay06}
D.~Kay and R.~Welford.
\newblock A multigrid finite element solver for the {Cahn-Hilliard} equation.
\newblock {\em J. Comput. Phys.}, 212:288--304, 2006.

\bibitem{kay07}
D.~Kay and R.~Welford.
\newblock Efficient numerical solution of {Cahn-Hilliard-Navier-Stokes} fluids
  in 2d.
\newblock {\em SIAM J. Sci. Comput.}, 29:2241--2257, 2007.

\bibitem{kim03}
J.S. Kim, K.~Kang, and J.S. Lowengrub.
\newblock Conservative multigrid methods for {C}ahn-{H}illiard fluids.
\newblock {\em J. Comput. Phys.}, 193:511--543, 2003.

\bibitem{larche82}
F.C. Larch\'{e} and J.W. Cahn.
\newblock The effect of self-stress on diffusion in solids.
\newblock {\em Acta Metall.}, 30:1835--1845, 1982.

\bibitem{lee02a}
H.G. Lee, J.S Lowengrub, and J.~Goodman.
\newblock Modeling pinchoff and reconnection in a {Hele-Shaw} cell. {I.} {The}
  models and their calibration.
\newblock {\em Physics of Fluids}, 14:492--513, 2002.

\bibitem{lee02b}
H.G. Lee, J.S Lowengrub, and J.~Goodman.
\newblock Modeling pinchoff and reconnection in a {Hele-Shaw} cell. {II.}
  {Analysis} and simulation in the nonlinear regime.
\newblock {\em Physics of Fluids}, 14:514--545, 2002.

\bibitem{liu03}
C.~Liu and J.~Shen.
\newblock A phase field model for the mixture of two incompressible fluids and
  its approximation by a {Fourier}-spectral method.
\newblock {\em Physica D}, 179:211--228, 2003.

\bibitem{fenics12}
A.~Logg, K.A. Mardal, and G.~N. Wells.
\newblock {\em Automated Solution of Differential Equations by the Finite
  Element Method: The {FEniCS} Book}.
\newblock Springer Verlag, Berlin, 2012.

\bibitem{shen12}
J.~Shen, C.~Wang, X.~Wang, and S.M. Wise.
\newblock Second-order convex splitting schemes for gradient flows with
  ehrlich-schwoebel type energy: Application to thin film epitaxy.
\newblock {\em SIAM Journal on Numerical Analysis}, 50(1):105--125, 2012.

\bibitem{shen10a}
J.~Shen and X.~Yang.
\newblock Numerical approximations of {Allen-Cahn} and {Cahn-Hilliard}
  equations.
\newblock {\em Discrete Contin. Dyn. Sys. A}, 28:1669--1691, 2010.

\bibitem{shen10b}
J.~Shen and X.~Yang.
\newblock A phase-field model and its numerical approximation for two-phase
  incompressible flows with different densities and viscosities.
\newblock {\em SIAM J. Sci. Comput.}, 32:1159--1179, 2010.

\bibitem{wise10}
S.M. Wise.
\newblock Unconditionally stable finite difference, nonlinear multigrid
  simulation of the {Cahn-Hilliard-Hele-Shaw} system of equations.
\newblock {\em J. Sci. Comput.}, 44:38--68, 2010.

\bibitem{wise05}
S.M. Wise, J.S. Lowengrub, J.S. Kim, K.~Thornton, P.W. Voorhees, and W.C.
  Johnson.
\newblock Quantum dot formation on a strain-patterned epitaxial thin film.
\newblock {\em Appl. Phys. Lett.}, 87:133102, 2005.

\bibitem{wise09a}
S.M. Wise, C.~Wang, and J.~Lowengrub.
\newblock An energy stable and convergent finite-difference scheme for the
  phase field crystal equation.
\newblock {\em SIAM J. Numer. Anal.}, 47:2269--2288, 2009.

\bibitem{wu14}
X.~Wu, G.J. van Zwieten, and K.G. van~der Zee.
\newblock Stabilized second-order convex splitting schemes for {Cahn-Hilliard}
  models with application to diffuse-interface tumor-growth models.
\newblock {\em International Journal for Numerical Methods in Biomedical
  Engineering}, 30:180--203, 2014.

\end{thebibliography}
	
	\end{document}